\documentclass[12pt]{amsart}

\title[Satisfaction is not absolute]{Satisfaction is not absolute}

\author[Hamkins]{Joel David Hamkins}
\address[Joel David Hamkins]
{O'Hara Professor of Logic, University of Notre Dame, 100 Malloy Hall, Notre Dame, IN 46556 USA}
\email{jdhamkins@nd.edu}
\urladdr{https://jdh.hamkins.org}

\author[Yang]{Ruizhi Yang}
 \address[R. Yang]
         {School of Philosophy,
         Fudan University,
         220 Handan Road, Shanghai, 200433 China}
 \email{yangruizhi@fudan.edu.cn}

\thanks{This work was begun in June 2013 and largely completed at that time. For inexplicable reasons, despite that it was accepted by the journal subject to minor revisions, the authors became distracted with other projects and this paper languished until 2025. We are now pleased finally to bring the paper to completion. The first author is grateful for the support provided to him as a visitor to Fudan University in Shanghai in June 2013 and again in July 2025. The second author's research has been supported by the Major Project of the Key Research Institute of Humanities and Social Sciences under the Ministry of Education (Project No. 22JJD110002). He is thankful to the Institute for Mathematical Sciences at the National University of Singapore for continuing support for him to attend the logic summer school every summer for 4 years, where the authors first met. The authors are both thankful to Roman Kossak for fruitful discussions, to Andrew Marks for proving lemma \ref{Lemma.SubsetOfPlane}, and to W. Hugh Woodin for suggesting a simplification in one of our arguments. Commentary concerning this article can be made at http://jdh.hamkins.org/satisfaction-is-not-absolute.}

\usepackage{amssymb}
\usepackage[hidelinks]{hyperref}
\usepackage[dvipsnames,svgnames]{xcolor}
\usepackage{tikz} % for diagrams and drawing
\usetikzlibrary{arrows,arrows.meta,calc}
\newtheorem{theorem}{Theorem}
\newtheorem*{theorem*}{Theorem}

\newtheorem*{maintheorem*}{Main Theorem}
\newtheorem*{maintheorems*}{Main Theorems}
\newtheorem{corollary}[theorem]{Corollary}
\newtheorem*{corollary*}{Corollary}
\newtheorem*{corollaries*}{Corollaries}
\newtheorem{sublemma}{Lemma}[theorem]

\newtheorem{proposition}[theorem]{Proposition}

\theoremstyle{definition}

\newtheorem*{definition*}{Definition}

\newtheorem*{question*}{Question}

\newtheorem*{questions*}{Questions}
\newtheorem*{mainquestion*}{Main Question} % without numbering
\newtheorem*{openquestion*}{Open Question} % without numbering
\theoremstyle{remark}

\theoremstyle{plain}% in case new theorems are declared in main.
\newcommand{\QED}{\end{proof}}

\def\proclaim[#1]{{\bf #1}}
\def\BF#1.{{\bf #1.}}

\def\says#1:#2\par{\item[#1] #2\par}
\newcommand{\Godel}{G\"odel}

\newcommand{\Levy}{L\'{e}vy}

\newcommand{\N}{{\mathbb N}}

\newcommand{\Z}{{\mathbb Z}}
\newcommand{\R}{{\mathbb R}}

\makeatletter
\newcommand{\dotminus}{\mathbin{\text{\@dotminus}}}
\newcommand{\@dotminus}{%
  \ooalign{\hidewidth\raise1ex\hbox{.}\hidewidth\cr$\m@th-$\cr}%
}
\makeatother

\newcommand{\of}{\subseteq}

\newcommand{\set}[1]{\{\,{#1}\,\}}

\newcommand{\elesub}{\prec}

\newcommand{\Th}{\mathop{\rm Th}}
\newcommand{\Con}{\mathop{{\rm Con}}}

\newcommand{\satisfies}{\models}

\newcommand{\proves}{\vdash}
\newcommand\dbrace{\hskip-1.5em\raise3pt\hbox{\rotatebox[origin=c]{-35}{$\left.\strut^{\phantom{|}}\right\}$}}}% useful for tetration

\newcommand\UParroW{{\setbox0\hbox{$\Uparrow$}\rlap{\hbox to \wd0{\hss$\mid$\hss}}\box0}}
\renewcommand{\setminus}{\raise.3ex\hbox{\rotatebox{-20}{$-$}}} % the usual setminus is absurdly huge and vertical

\newcommand{\Union}{\bigcup}

\newcommand{\trianglelt}{\vartriangleleft}

\newcommand{\smalllt}{\mathrel{\mathchoice{\raise2pt\hbox{$\scriptstyle<$}}{\raise1pt\hbox{$\scriptstyle<$}}{\raise0pt\hbox{$\scriptscriptstyle<$}}{\scriptscriptstyle<}}}
\newcommand{\smallleq}{\mathrel{\mathchoice{\raise2pt\hbox{$\scriptstyle\leq$}}{\raise1pt\hbox{$\scriptstyle\leq$}}{\raise1pt\hbox{$\scriptscriptstyle\leq$}}{\scriptscriptstyle\leq}}}
\newcommand{\lt}{\smalllt}
\newcommand{\ltomega}{{{\smalllt}\omega}}

   \def\DHLhksqrt#1#2{%
   \setbox0=\hbox{$#1\sqrt{#2\,}$}\dimen0=\ht0
   \advance\dimen0-0.2\ht0
   \setbox2=\hbox{\vrule height\ht0 depth -\dimen0}%
   {\box0\lower0.4pt\box2}}

\def\[#1]{\mathopen{\lbrack\!\lbrack}#1\mathclose{\rbrack\!\rbrack}}
\newbox\gnBoxA
\newbox\gnBoxB
\newdimen\gnCornerHgt
\setbox\gnBoxA=\hbox{\tiny$\ulcorner$}
\global\gnCornerHgt=\ht\gnBoxA
\newdimen\gnArgHgt
\def\gcode #1{%
\setbox\gnBoxA=\hbox{$#1$}%
\setbox\gnBoxB=\hbox{$\bar #1$}%
\gnArgHgt=\ht\gnBoxB%
\ifnum     \gnArgHgt<\gnCornerHgt \gnArgHgt=0pt%
\else \advance \gnArgHgt by -\gnCornerHgt%
\fi \raise\gnArgHgt\hbox{\tiny$\ulcorner$} \box\gnBoxA %
\raise\gnArgHgt\hbox{\tiny$\urcorner$}}
\newcommand{\UnderTilde}[1]{{\setbox1=\hbox{$#1$}\baselineskip=0pt\vtop{\hbox{$#1$}\hbox to\wd1{\hfil$\sim$\hfil}}}{}}
\newcommand{\Undertilde}[1]{{\setbox1=\hbox{$#1$}\baselineskip=0pt\vtop{\hbox{$#1$}\hbox to\wd1{\hfil$\scriptstyle\sim$\hfil}}}{}}
\newcommand{\undertilde}[1]{{\setbox1=\hbox{$#1$}\baselineskip=0pt\vtop{\hbox{$#1$}\hbox to\wd1{\hfil$\scriptscriptstyle\sim$\hfil}}}{}}
\newcommand{\UnderdTilde}[1]{{\setbox1=\hbox{$#1$}\baselineskip=0pt\vtop{\hbox{$#1$}\hbox to\wd1{\hfil$\approx$\hfil}}}{}}
\newcommand{\Underdtilde}[1]{{\setbox1=\hbox{$#1$}\baselineskip=0pt\vtop{\hbox{$#1$}\hbox to\wd1{\hfil\scriptsize$\approx$\hfil}}}{}}

\renewcommand{\th}{{\hbox{\scriptsize th}}}
\renewcommand{\implies}{\mathrel{\rightarrow}}

\renewcommand{\iff}{\mathrel{\leftrightarrow}}

\newcommand{\iso}{\cong}
\def\<#1>{\left\langle#1\right\rangle}

\newcommand{\Tr}{\mathop{\rm Tr}\nolimits}

\newcommand{\Ord}{\mathord{{\rm Ord}}}
\newcommand{\WO}{\mathord{{\rm WO}}}

\newcommand{\HC}{\mathop{{\rm HC}}}

\newcommand{\ZFC}{{\rm ZFC}}
\newcommand{\ZF}{{\rm ZF}}

\newcommand{\CH}{{\rm CH}}

\newcommand{\HF}{{\rm HF}}

\newcommand{\PA}{{\rm PA}}

\newcommand{\TA}{{\rm TA}}

\newcommand{\cell}[1]{\boxit{\hbox to 17pt{\strut\hfil$#1$\hfil}}}
\newcommand{\head}[2]{\lower2pt\vbox{\hbox{\strut\footnotesize\it\hskip3pt#2}\boxit{\cell#1}}}
\newcommand{\boxit}[1]{\setbox4=\hbox{\kern2pt#1\kern2pt}\hbox{\vrule\vbox{\hrule\kern2pt\box4\kern2pt\hrule}\vrule}}
\newcommand{\Col}[3]{\hbox{\vbox{\baselineskip=0pt\parskip=0pt\cell#1\cell#2\cell#3}}}
\newcommand{\tapenames}{\raise 5pt\vbox to .7in{\hbox to .8in{\it\hfill input: \strut}\vfill\hbox to
.8in{\it\hfill scratch: \strut}\vfill\hbox to .8in{\it\hfill output: \strut}}}
\newcommand{\Head}[4]{\lower2pt\vbox{\hbox to25pt{\strut\footnotesize\it\hfill#4\hfill}\boxit{\Col#1#2#3}}}
\newcommand{\Dots}{\raise 5pt\vbox to .7in{\hbox{\ $\cdots$\strut}\vfill\hbox{\ $\cdots$\strut}\vfill\hbox{\
$\cdots$\strut}}}
\newcommand{\df}{\it} % use italic for definition terms. Idea: also use this to create an index of definitions, if MakeIndex is true.
\usepackage{mathrsfs}
\newcommand{\Sat}{\mathord{\rm Sat}}

\begin{document}

\begin{abstract}
We prove that the satisfaction relation $\mathcal{N}\satisfies\varphi[\vec a]$ of first-order logic is not absolute between models of set theory having the structure $\mathcal{N}$ and the formulas $\varphi$ all in common. Two models of set theory can have the same natural numbers, for example, and the same standard model of arithmetic $\<\N,{+},{\cdot},0,1,{\lt}>$, yet disagree on their theories of arithmetic truth; two models of set theory can have the same natural numbers and the same arithmetic truths, yet disagree on their truths-about-truth, at any desired level of the iterated truth-predicate hierarchy; two models of set theory can have the same natural numbers and the same reals, yet disagree on projective truth; two models of set theory can have the same $\<H_{\omega_2},{\in}>$ or the same rank-initial segment $\<V_\delta,{\in}>$, yet disagree on which assertions are true in these structures.

On the basis of these mathematical results, we argue that a philosophical commitment to the determinateness of the theory of truth for a structure cannot be seen as a consequence solely of the determinateness of the structure in which that truth resides. The determinate nature of arithmetic truth, for example, is not a consequence of the determinate nature of the arithmetic structure $\N=\set{0,1,2,\ldots}$ itself, but rather, we argue, is an additional higher-order commitment requiring its own analysis and justification.
\end{abstract}

\maketitle

\section{Introduction}

Many mathematicians and philosophers regard the natural numbers $0,1,2,\ldots\,$, along with their usual arithmetic structure, as having a privileged mathematical existence, a Platonic realm of numbers in which assertions have definite, absolute truth values, independently of our ability to prove or discover them. Although there are some arithmetic assertions that we can neither prove nor refute---such as the consistency of the background theory in which we undertake our proofs---the view is that nevertheless there is a fact of the matter about whether any such arithmetic statement is true or false in the intended interpretation. The definite nature of arithmetic truth is often seen as a consequence of the definiteness of the structure of arithmetic $\<\N,{+},{\cdot},0,1,{\lt}>$ itself, for if the natural numbers exist in a clear and distinct totality in a way that is unambiguous and absolute, then (on this view) the first-order theory of truth residing in that structure---arithmetic truth---is similarly clear and distinct.

Feferman provides an instance of this perspective when he writes:
\begin{quote}
 In my view, the conception [of the bare structure of the natural numbers] is {\it completely clear}, and thence {\it all arithmetical statements are definite}. \cite[p.6--7]{Feferman:CommentsForEFIWorkshop} (emphasis original)
\end{quote}
It is Feferman's `thence' to which we call attention, and many mathematicians and philosophers seem to share this perspective. Martin writes:
\begin{quote}
What I am suggesting is that the real reason for confidence in first-order completeness is our confidence in the full determinateness of the concept of the natural numbers. \cite[p. 13]{Martin:CompletenessOrIncompletenessOfBasicMathematicalConcepts}
\end{quote}
The truth of an arithmetic statement, to be sure, does seem to depend entirely on the structure $\<\N,{+},{\cdot},0,1,{\lt}>$, with all quantifiers restricted to $\N$ and using only those arithmetic operations and relations, and so if that structure has a ``definite'' nature, then it would seem that the truth of the statement should be similarly definite.

Nevertheless, in this article we should like to tease apart these two ontological commitments, arguing that the definiteness of truth for a given mathematical structure, such as the natural numbers, the reals or higher-order structures such as $H_{\omega_2}$ or $V_\delta$, does not follow from the definite nature of the underlying structure in which that truth resides. Rather, we argue that the commitment to a theory of truth for a structure is a higher-order ontological commitment, going strictly beyond the commitment to a definite nature for the underlying structure itself. This is, of course, a commitment that we expect many of those authors to want to make---our main point is that it does not come for free.

We shall make our argument by first proving, as a strictly mathematical matter, that different models of set theory can have a structure identically in common, even the natural numbers, yet disagree on the theory of truth for that structure.
\begin{itemize}
  \item Models of set theory can have the same structure of arithmetic $\<\N,{+},{\cdot},0,1,{\lt}>$, yet disagree on arithmetic truth.
  \item Models of set theory can have the same reals, yet disagree on projective truth.
  \item Models of set theory can have a transitive rank initial segment $V_\delta$ in common, yet disagree about whether it is a model of \ZFC.
\end{itemize}
In these cases and many others, the theory of a structure is not absolute between models of set theory having that structure identically in common. This is a stronger kind of non-absoluteness phenomenon than the usual observation, via the incompleteness theorem, that models of set theory can disagree on arithmetic truth, for here we have models of set theory, which disagree about arithmetic truth, yet agree completely on the structure in which that truth resides.  Our mathematical claims will be made in sections \ref{Section.IndefiniteArithmeticTruth} through \ref{Secton.BeingModelZFCNotAbsolute}. Afterwards, on the basis of these mathematical observations, we shall draw our philosophical conclusions in section \ref{Section.Conclusions}.

{\it Satisfaction is absolute!} This slogan, heard by the first author in his graduate-student days---a fellow logic student would assert it with exaggerated double-entendre---was meant to evoke the idea that the satisfaction relation $\mathcal{N}\satisfies\varphi[\vec a]$ is absolute between all the various models of set theory able to express it. The set-theoretic universe $V$, for example, has the same arithmetic truths as the constructible universe $L$; and it doesn't matter, when asking whether $\mathcal{N}\satisfies\varphi[\vec a]$, whether one determines the answer in the universe $V$ or in a forcing extension $V[G]$ of it. So the slogan is true to a very high degree, such as in the case of Shoenfield absoluteness and \Levy\ absoluteness, and in particular, it is true whenever the formula $\varphi$ has standard-finite length in the meta-theory or between any two models of set theory for which at least one has access to the satisfaction relation of the other, since any model of set theory that can see two satisfaction relations for a structure will see by induction on formulas that they must agree. Nevertheless, the main theorems of this article show that the slogan is not true when one uses a broader concept of absoluteness, namely, satisfaction is {\it not} absolute, for there can be models of set theory with a structure $\mathcal{N}$ and sentence $\sigma$ in common, but which disagree about whether $\mathcal{N}\satisfies\sigma$.

Before proceeding further, we should like to remark on the folklore nature of some of the mathematical arguments and results contained in sections \ref{Section.IndefiniteArithmeticTruth} through \ref{Secton.BeingModelZFCNotAbsolute} of this article. Theorems \ref{Theorem.SameNDifferentTruths} and \ref{Theorem.SameModelNDifferentS} and their consequences, for example, are proved here using only well-known classical methods. These arguments could be considered as a part of the mathematical folklore of the subject of models of arithmetic, a subject filled with many fascinating results about automorphisms of nonstandard models of arithmetic and of set theory and the images of non-definable sets in computably saturated models. For example, Schlipf \cite{Schlipf1978:TowardModelTheoryThroughRecursiveSaturation} proved many basic results about computably saturated models, including the case of $\<\mathcal{M},X>$ where $X$ is not definable in $M$, a case which figures in our theorems \ref{Theorem.SameNDifferentTruths} and \ref{Theorem.SameModelNDifferentS}. For example, Schlipf proves that if $\mathcal{M}\satisfies\ZF$ is resplendent, then there is a cofinal set of indiscernibles $I$ in the ordinals of $\mathcal{M}$ such that for each $\alpha\in I$ we have $V_\alpha^{\mathcal{M}}\elesub \mathcal{M}$ and $V_\alpha^{\mathcal{M}}\iso\mathcal{M}$; if $\mathcal{M}$ is also countable and $V_\alpha^{\mathcal{M}}\elesub\mathcal{M}$, then there are $2^\omega$ many distinct isomorphisms $V_\alpha^{\mathcal{M}}\cong\mathcal{M}$; and  $\mathcal{M}$ is isomorphic to some topless initial segment $\mathcal{N}\elesub\mathcal{M}$, that is, for which $\mathcal{M}$ has no supremum to the ordinals of $\mathcal{N}$. Kossak and Kotlarski \cite{KossakKotlarski1988:ResultsOnAutomorphismsOfRecursivelySaturatedModelsOfPA} identified circumstances under which a nondefinable subset $X$ of a countable model $\mathcal{M}\satisfies\PA$ must have the maximum number of automorphic images in $\mathcal{M}$, including the case where $X$ is an inductive satisfaction class. Schmerl subsequently proved that every undefinable class $X$ in such a model $\mathcal{M}$ has continuum many automorphic images. In other work, Kossak and Kotlarski \cite{KossakKotlarski92:GameApproximationsOfSatisfactionClassesAndTheProblemOfRatherClasslessModels} proved that if $\mathcal{M}$ is a model of \PA\ with a full inductive satisfaction class, then it has full inductive satisfaction classes $S_1$ and $S_2$ which disagree on a set of sentences that is coinitial with the standard cut. So the topic is well-developed and much is known. Concerning the specific results we prove in this article, experts in the area seem instinctively to want to prove them by means of resplendency and the other sophisticated contemporary ideas that frame the current understanding of the subject---showing the depth and power of those methods---and indeed one may prove the theorems via resplendency. Nevertheless, our arguments here show that elementary methods suffice.

\section{Indefinite arithmetic truth}\label{Section.IndefiniteArithmeticTruth}

Let us begin with what may seem naively to be a surprising case, where we have two models of set theory with the same structure of arithmetic $\<\N,{+},{\cdot},0,1,{\lt}>$, but different theories of arithmetic truth.

\begin{theorem}\label{Theorem.SameNDifferentTruths}
 Every consistent extension of \ZFC\ has two models $M_1$ and $M_2$, which agree on the natural numbers and on the structure $\<\N,{+},{\cdot},0,1,{\lt}>^{M_1}=\<\N,{+},{\cdot},0,1,{\lt}>^{M_2}$, but which disagree their theories of arithmetic truth, in the sense that there is in $M_1$ and $M_2$ an arithmetic sentence $\sigma$, such that $M_1$ thinks $\sigma$ is true, but $M_2$ thinks it is false.
$$
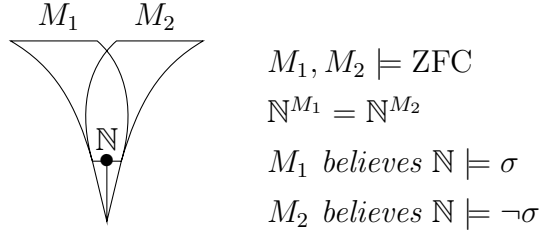
\begin{tikzpicture}[scale=.8,xscale=.8]
 \node at (0,1) (N) {$\bullet$};
 \node[anchor=south,scale=.8] at (N) {$\N$};
 \draw (0,1) --(0,0) --(-.3,1) to [out=110, in=-25](-2,3) --(-.2,3) to [out=-35,in=70](.3,1) --(-.3,1) to [out=110,in=-145] (.2,3) --(2,3) to [out=-155,in=70](.3,1) --(0,0);
 \node[anchor=south] at (-1,3) {$M_1$};
 \node[anchor=south] at (1,3) {$M_2$};
\end{tikzpicture}
\qquad
\vbox{\hbox{$M_1,M_2\satisfies\ZFC$}\medskip\hbox{$\N^{M_1}=\N^{M_2}$}\medskip\hbox{$M_1$ believes $\N\satisfies\sigma$}\medskip\hbox{$M_2$ believes $\N\satisfies\neg\sigma$}}
$$
\end{theorem}

Thus, two models of set theory can agree on which natural numbers exist and agree on all the details of the standard model of arithmetic, yet disagree on which sentences are true in that model. The proof is elementary, but before giving the proof, we should like to place the theorem into the context of some classical results on arithmetic truth, particularly Krajewski's work \cite{Krajewski1974:MutuallyInconsistentSatisfactionClasses,Krajewski1976:Non-standardSatisfactionClasses} on incompatible satisfaction classes, explained in theorem \ref{Theorem.Krajewski}.

Inside every model of set theory $M\satisfies\ZFC$, we may extract a canonical model of arithmetic, the structure $\<\N,{+},{\cdot},0,1,{\lt}>^M$, which we henceforth denote simply by $\N^M$, arising from what $M$ views as the standard model of arithmetic. Namely, $\N^M$ is the structure whose objects are the objects that $M$ thinks to be natural numbers and whose operations and relations agree with what $M$ thinks are the standard arithmetic operations and relations on the natural numbers. Let us define that a {\df \ZFC-standard} model of arithmetic, or just a {\df standard} model of arithmetic (as opposed to \emph{the} standard model of arithmetic), is a model of arithmetic that arises in this way as $\N^M$ for some model $M\satisfies\ZFC$. In other words, {\it a} standard model of arithmetic is one that is thought to be {\it the} standard model of arithmetic from the perspective of some model of \ZFC. More generally, for any set theory $T$ we say that a model of arithmetic is a {\df $T$-standard} model of arithmetic, if it arises as $\N^M$ for some $M\satisfies T$.

Every model of set theory $M\satisfies\ZFC$ has what it thinks is the true theory of arithmetic $\TA^M$, the collection of $\sigma$ thought by $M$ to be (the \Godel\ code of) a sentence in the language of arithmetic, true in $\N^M$. (In order to simplify notation, we shall henceforth identify formulas with their \Godel\ codes.) The theory $\TA^M$ is definable in $M$ by means of the recursive Tarskian definition of truth-in-a-structure, although it is not definable in $\N^M$, by Tarski's theorem on the non-definability of truth. Note that when $\N^M$ is nonstandard, this theory will include many nonstandard sentences $\sigma$, which do not correspond to any actual assertion in the language of arithmetic from the perspective of the metatheory, but nevertheless, these sentences gain a meaningful truth value inside $M$, where they appear to be standard, via the Tarski recursion as carried out inside $M$. In this way, the \ZFC-standard models of arithmetic can be equipped with a notion of truth, which obeys the recursive requirements of the Tarskian definition.

These relativized truth predicates are instances of the more general concept of a satisfaction class for a model of arithmetic (see \cite{KossakSchmerl2006:TheStructureOfModelsOfPA} and \cite{Kaye1991:ModelsOfPeanoArithemtic} for general background). Given a model of arithmetic $\mathcal{N}=\<N,{+},{\cdot},0,1,{\lt}>$, we say that a subclass $\Tr\of N$ is a {\df truth predicate} or a {\df full satisfaction class} for $\mathcal{N}$---we shall use the terms interchangeably---if every element of $\Tr$ is a sentence in the language of arithmetic as viewed by $\mathcal{N}$, and such that $\Tr$ obeys the recursive Tarskian definition of truth:
\begin{enumerate}
 \item[1.] (atomic)  For each atomic sentence $\sigma$ in $N$, we have $\sigma\in \Tr$ just in case $N$ thinks $\sigma$ is true. (Note that the value of any closed term may be uniquely evaluated inside $\mathcal{N}$ by an internal recursion, and so every model of arithmetic has a definable relation for determining the truth of atomic assertions.)
 \item[2.] (conjunction) $\sigma\wedge\tau\in \Tr$ if and only if $\sigma\in \Tr$ and $\tau\in \Tr$.
 \item[3.] (negation) $\neg\sigma\in \Tr$ if and only if $\sigma\notin \Tr$.
 \item[4.] (quantifiers) $\exists x\, \varphi(x)\in \Tr$ if and only if there is some $n\in N$ such that $\varphi(\bar n)\in \Tr$, where $\bar n$ is the corresponding term $\underbrace{1+\cdots+1}_n$ as constructed in $\mathcal{N}$.
\end{enumerate}
Note that $\Tr$ is applied only to assertions in the language of arithmetic, not to assertions in the expanded language using the truth predicate itself (but we look at iterated truth predicates in section \ref{Section.SatisfactionIsNotAbsolute}). A truth predicate $\Tr$ for a model $\mathcal{N}=\<N,{+},{\cdot},0,1,{\lt}>$ is {\df inductive}, if the expanded structure $\<N,{+},{\cdot},0,1,{\lt},\Tr>$ satisfies $\PA(\Tr)$, the theory of \PA\ in the language augmented with a predicate symbol for $\Tr$, so that mentions of $\Tr$ may appear in instances of the induction axiom. The theory of true arithmetic $\TA^M$ arising in any model of set theory $M$ is easily seen to be an inductive truth predicate for the corresponding \ZFC-standard model of arithmetic $\N^M$ arising in that model of set theory, simply because \ZFC\ proves that $\N$ satisfies the second-order Peano axioms, and so $\ZFC$ proves that every subset of $\N$ is inductive.

A principal case of Tarski's theorem on the non-definability of truth is the fact that no model of arithmetic $\mathcal{N}$ can have a truth predicate that is definable in the language of arithmetic. This is simply because for every arithmetic formula $\tau(x)$ there is by the \Godel\ fixed-point lemma a sentence $\sigma$ such that $\PA\proves\sigma\iff\neg\tau(\sigma)$, and so we would have either that $\sigma$ is true in $\mathcal{N}$ while $\tau(\sigma)$ fails, or that $\sigma$ is false in $\mathcal{N}$ while $\tau(\sigma)$ holds; either of these possibilities would mean that the collection of sentences satisfying $\tau$ could not satisfy the recursive Tarskian truth requirements, applied up to the logical complexity of $\sigma$, which is a standard-finite sentence.

Krajewski observed that a model of arithmetic can have different incompatible truth predicates, a fact we find illuminating for the context of this paper, and so we presently give an account of it. The argument is pleasantly classical, relying principally only on Beth's implicit definability theorem and Tarski's theorem on the non-definability of truth.

\begin{theorem}[\cite{Krajewski1974:MutuallyInconsistentSatisfactionClasses,Krajewski1976:Non-standardSatisfactionClasses}]\label{Theorem.Krajewski}
 There are models of arithmetic with different incompatible inductive truth predicates. Indeed, every model of arithmetic $N_0\satisfies\PA$ that admits an (inductive) truth predicate has an elementary extension $N$ that admits several incompatible (inductive) truth predicates.
\end{theorem}

\begin{proof}
Let $\mathcal{N}_0=\<N_0,{+},{\cdot},0,1,{\lt}>$ be any model of arithmetic that admits a truth predicate (for example, the standard model $\N^M$ arising in any  model of set theory $M$). Let $T$ be the theory consisting of the elementary diagram $\Delta(\mathcal{N}_0)$ of this model, in the language of arithmetic with constants for every element of $N_0$, together with the assertion ``$\Tr$ is a truth predicate,'' which is expressible as a single assertion about $\Tr$, namely, the assertion that it satisfies the recursive Tarskian truth requirements. The theory $T$ is consistent, because by assumption, $\mathcal{N}_0$ itself admits a truth predicate. Furthermore, any model of the theory $T$ provides an elementary extension $\mathcal{N}$ of $\mathcal{N}_0$, when reduced to the language of arithmetic, together with a truth predicate for $\mathcal{N}$. Suppose toward contradiction that every elementary extension of $\mathcal{N}_0$ that admits a truth predicate has a unique such class. It follows that any two models of $T$ with the same reduction to the language of the diagram of $\mathcal{N}_0$ must have the same interpretation for the predicate $\Tr$. Thus, the predicate $\Tr$ is implicitly definable in $T$, in the sense of the Beth implicit definability theorem (see \cite[thm 2.2.22]{ChangKeisler1990:ModelTheory}), and so by that theorem, the predicate $\Tr$ must be explicitly definable in any model of $T$ by a formula in the base language, the language of arithmetic with constants for elements of $\mathcal{N}_0$. But this violates Tarski's theorem on the non-definability of truth, which implies that no model of arithmetic can have a definable truth predicate. Thus, there must be models of $T$ with identical reductions $\mathcal{N}$ to the language of the diagram of $\mathcal{N}_0$, but different truth predicates on $\mathcal{N}$. In other words, $\mathcal{N}$ is an elementary extension of $\mathcal{N}_0$ having at least two different incompatible truth predicates. (And it is not difficult by similar reasoning to see that there must be an $\mathcal{N}$ having infinitely many distinct truth predicates.)

In the case where $\mathcal{N}_0$ admits an inductive full satisfaction class, we simply add $\PA(\Tr)$ to the theory $T$, with the result by the same reasoning that the elementary extension $\mathcal{N}$ will also have multiple inductive full satisfaction classes, as desired.
\end{proof}

In fact, Krajewski \cite{Krajewski1976:Non-standardSatisfactionClasses} proves that we may find elementary extensions $\mathcal{N}$ having at least any desired cardinal $\kappa$ many such full satisfaction classes.

We observe next the circumstances under which these various satisfaction classes can become the true theory of arithmetic inside a model of set theory. First, let's note the circumstances under which a model of arithmetic is a \ZFC-standard model of arithmetic. Let $\Th(\N)^{\ZFC}$ be the set of sentences $\sigma$ in the language of arithmetic, such that $\ZFC\proves(\<\N,{+},{\cdot},0,1,{\lt}>\satisfies\sigma)$. These are the arithmetic consequences of \ZFC, the sentences that hold in every \ZFC-standard model of arithmetic. If \ZFC\ is consistent, then so is this theory, since it holds in the standard model of arithmetic of any model of \ZFC. More generally, for any set theory $T$ proving the existence of the standard model $\N$ and its theory $\TA$, we have the theory $\Th(\N)^T$, consisting of the sentences $\sigma$ in the language of arithmetic that $T$ proves to hold in the standard model. Thanks to Roman Kossak (recently) and Ali Enayat (from some time ago) for discussions concerning the following proposition. The result appears in \cite{Enayat2009:OnZFSTandardModelsOfPA} with further related analysis of the \ZFC-standard models of arithmetic.

\begin{proposition}[Enayat]\label{Proposition.ZFCStandard}
 The following are equivalent for any countable nonstandard model of arithmetic $\mathcal{N}$.
 \begin{enumerate}
  \item $\mathcal{N}$ is a $\ZFC$-standard model of arithmetic. That is, $\mathcal{N}=\N^M$ for some $M\satisfies\ZFC$.
  \item $\mathcal{N}$ is a computably saturated model of $\Th(\N)^\ZFC$.
 \end{enumerate}
\end{proposition}

\begin{proof}
($1\to 2$) Suppose that $\mathcal{N}$ is a countable nonstandard \ZFC-standard model of arithmetic, arising as $\mathcal{N}=\N^M$ for some model of set theory $M\satisfies\ZFC$, which we may assume is countable. Clearly, $\mathcal{N}$ satisfies the theory $\Th(\N)^\ZFC$. Note also that $\mathcal{N}$ admits an inductive satisfaction class, namely, the collection of arithmetic truths $\TA^M$ as they are defined inside $M$. An easy overspill argument shows that any model of \PA\ with an inductive satisfaction class is computably saturated (e.g. \cite[lemma~1.1]{Smorynski1981:RecursivelySaturatedNonstandardModelsOfArithmetic}). Namely, suppose that $p(x,\vec n)$ is a computable type in the language of arithmetic with parameters $\vec n\in\mathcal{N}$ and that $p$ is finitely realized in $\mathcal{N}$. Since the type $p$ is computable, the model $\mathcal{N}$ computes its own version of the type $p^{\mathcal{N}}$, and this will agree with $p$ on all standard formulas. Furthermore, the structure $\<\mathcal{N},\TA^M>$ can see that all the standard-finite initial segments of $p^{\mathcal{N}}$ are satisfiable according to the truth predicate $\TA^M$, which by induction in the meta-theory agrees with actual truth on the standard-finite length formulas. Thus, by overspill (since we have induction for this predicate), it follows that some nonstandard length initial segment of $p^{\mathcal{N}}$ is satisfied in $\mathcal{N}$ according to the truth predicate, and since that truth predicate agrees with actual satisfaction on standard-finite formulas, it follows in particular that $p$ itself is satisfied in $\mathcal{N}$. So $\mathcal{N}$ is a countable computably saturated model of $\Th(\N)^\ZFC$, as desired.

($2\to 1$) Conversely, suppose that $\mathcal{N}$ is a countable computably saturated model of $\Th(\N)^\ZFC$. Consider the theory $T$ consisting of the \ZFC\ axioms together with the assertions $\sigma^\N$ for every sentence $\sigma\in\Th(\mathcal{N})$, which is consistent precisely because $\mathcal{N}\satisfies\Th(\N)^\ZFC$. Indeed, $\mathcal{N}$ agrees on the consistency of any particular finite subtheory of $T$, because for any finite fragment $\varphi$ of \ZFC\ and any arithmetic sentence $\sigma$, we have $\ZFC\proves \sigma^\N\implies \Con(\varphi+\sigma^\N)$ by the reflection theorem, and so this statement is true in $\mathcal{N}$, along with its antecedent $\sigma$, and so $\varphi+\sigma^\N$ is consistent in $\mathcal{N}$. 

By computable saturation, the theory $T$ is coded in $\mathcal{N}$, since we may write down a computable type for this, and so there is a nonstandard finite theory $t\in\mathcal{N}$ whose standard part is exactly $T$. 

Since as we have mentioned $\mathcal{N}$ thinks the finite subtheories of $T$ are consistent, by overspill we may assume by cutting down to a nonstandard initial segment if necessary that $t$ is consistent in $\mathcal{N}$. Inside $\mathcal{N}$, we can now build the canonical complete consistent Henkin theory $H$ extending $t$, and let $M\satisfies H$ be the corresponding Henkin model. In particular, $M\satisfies\ZFC$, since this is a part of $t$, and so $\N^M$ is a \ZFC-standard model of arithmetic and hence computably saturated. Note also that $\N^M$ has the same theory as $\mathcal{N}$, because this is part of $t$ and hence $H$. The structure $\mathcal{N}$ can construct an isomorphism from itself with an initial segment of $\N^M$, because for every $a\in\mathcal{N}$ it has a Henkin constant $\check a$ witnessing $\check a=\overbrace{1+\cdots+1}^a$ and it must be part of the theory that any $x<\check a$ is some $\check b$ for some $b<a$ in $\mathcal{N}$, since otherwise $\mathcal{N}$ would think $H$ is inconsistent. It follows that $\mathcal{N}$ and $\N^M$ have the same standard system. But any two countable computably saturated models of arithmetic with the same standard system are isomorphic, by the usual back-and-forth argument, and so $\mathcal{N}\cong\N^M$, showing that $\mathcal{N}$ is \ZFC-standard, as desired.\end{proof}

Thus, the countable nonstandard \ZFC-standard models of arithmetic are precisely the countable computably saturated models of $\Th(\N)^\ZFC$. One may similarly show that every uncountable model of $\Th(\N)^\ZFC$ has an elementary extension to a \ZFC-standard model of arithmetic, since the theory $\ZFC+\set{\sigma^\N\mid\sigma\in\Delta(\mathcal{N})}$ is finitely consistent, where $\Delta(\mathcal{N})$ refers to the elementary diagram of $\mathcal{N}$ in the language with constants for every element of $\mathcal{N}$, and any model $M$ of this theory will have $\N^M$ as an elementary extension of $\mathcal{N}$, as desired. The equivalence stated in proposition \ref{Proposition.ZFCStandard} does not generalize to uncountable models, for there are uncountable computably saturated models of $\Th(\N)^\ZFC$ that are $\omega_1$-like and rather classless, which means in particular that they admit no inductive truth predicates and therefore are not \ZFC-standard.

We shall now extend proposition \ref{Proposition.ZFCStandard} to the case where the model carries a truth predicate. Let $\Th(\N,\TA)^\ZFC$ be the theory consisting of all sentences $\sigma$ for which $\ZFC\proves(\<\N,{+},{\cdot},0,1,{\lt},\TA>\satisfies\sigma)$, where $\sigma$ is in the language of arithmetic augmented with a truth predicate and \TA\ refers to the \ZFC-definable set of true arithmetic assertions. If \ZFC\ is consistent, then so is $\Th(\N,\TA)^\ZFC$, since it holds in the standard model of arithmetic, with the standard interpretation of \TA, arising inside any model of \ZFC. Note that $\Th(\N,\TA)^\ZFC$ includes the assertion that $\TA$ is a truth predicate, as well as the induction scheme in the language with this predicate.

\begin{proposition}\label{Proposition.ZFCStandardTA}\
 The following are equivalent for any countable nonstandard model of arithmetic $\mathcal{N}$ with a truth predicate $\Tr$.
 \begin{enumerate}
  \item $\<\mathcal{N},\Tr>$ is a $\ZFC$-standard model of arithmetic and arithmetic truth. That is, $\mathcal{N}=\N^M=\<\N,{+},{\cdot},0,1,{\lt}>^M$ for some $M\satisfies\ZFC$ in which $\Tr=\TA^M$ is the theory of true arithmetic.
  \item $\<\mathcal{N},\Tr>$ is a computably saturated model of $\Th(\N,\TA)^\ZFC$.
 \end{enumerate}
\end{proposition}

\begin{proof}
The proof of proposition \ref{Proposition.ZFCStandard} adapts to accommodate the expanded structure. If a countable nonstandard model $\<\mathcal{N},\Tr>$ arises as $\<\N,\TA>^M$ for some $M\satisfies\ZFC$, then it admits an inductive truth predicate in the expanded language, and this implies that it is computably saturated just as above. Conversely, any countably computably saturated model of $\Th(\N,\TA)^\ZFC$ can build as before a model of the corresponding Henkin theory extending $\ZFC+\set{\sigma^{\<\N,\TA>}\mid\<\mathcal{N},\Tr>\satisfies\sigma}$. The corresponding Henkin model $M$ will have $\<\mathcal{N},\Tr>$ as an initial segment of $\<\N^M,\TA>$, and so these two models have the same standard system, and since they also are elementarily equivalent and computably saturated, they are isomorphic by the back-and-forth construction. So $\<\mathcal{N},\Tr>$ is \ZFC-standard.
\end{proof}

\noindent Let us now finally prove theorem \ref{Theorem.SameNDifferentTruths}. We shall give two proofs, one as a corollary to proposition \ref{Proposition.ZFCStandardTA}, and another simpler direct proof.

\begin{proof}[Proof of theorem \ref{Theorem.SameNDifferentTruths}] (as corollary to proposition \ref{Proposition.ZFCStandardTA}) By Beth's theorem as in the proof of theorem \ref{Theorem.Krajewski}, we may find two different truth predicates on the same model of arithmetic, with both $\<\mathcal{N},\Tr_1>$ and $\<\mathcal{N},\Tr_2>$ being computably saturated models of $\Th(\N,\TA)^\ZFC$. It follows by proposition \ref{Proposition.ZFCStandardTA} that they arise as the standard model inside two different models of set theory, with $\mathcal{N}=\N^{M_1}=\N^{M_2}$ and $\Tr_1=\TA^{M_1}$ and $\Tr_2=\TA^{M_2}$, establishing theorem \ref{Theorem.SameNDifferentTruths}.
\end{proof}

We also give a simpler direct proof, as follows. (Thanks to W. Hugh Woodin for pointing out a further simplification in this argument.)

\begin{proof}[Proof of theorem \ref{Theorem.SameNDifferentTruths}] (direct argument) Suppose that $M_1$ is any countable $\omega$-nonstandard model of set theory. It follows that $M_1$'s version of the standard model of arithmetic $\<\N,{+},{\cdot},0,1,{\lt},\TA>^{M_1}$, augmented with what $M_1$ thinks is the true theory of arithmetic, is countable and computably saturated. Since $\TA^{M_1}$ is not definable in the reduced structure $\N^{M_1}=\<\N,{+},{\cdot},0,1,{\lt}>^{M_1}$, it follows by saturation that there is an arithmetic sentence $\sigma\in\TA^{M_1}$ realizing the same $1$-type in $\N^{M_1}$ as another sentence $\tau$ in $\N^{M_1}$ with $\tau\notin\TA^{M_1}$; that is, $\N^{M_1}\satisfies\varphi(\sigma)\iff\varphi(\tau)$ for every formula $\varphi$ in the language of arithmetic. To see this, consider the type $p(s,t)$ containing all formulas $\varphi(s)\iff\varphi(t)$ for $\varphi$ in the language of arithmetic, plus the assertions $s\in\TA$ and $t\notin\TA$. This is a computable type, and it is finitely realized in $\<\N,{+},{\cdot},0,1,{\lt},\TA>^{M_1}$, precisely because otherwise we would be able to define $\TA^{M_1}$ in $\N^{M_1}$. Thus, there are arithmetic sentences $\sigma$ and $\tau$ in $M_1$, whose \Godel\ codes realize the same $1$-type in $\N^{M_1}$, but such that $M_1$ thinks that $\sigma$ is true and $\tau$ is false in $\N^{M_1}$. Since these objects have the same $1$-type in $\N^{M_1}$, it follows by the back-and-forth construction that there is an automorphism $\pi:\N^{M_1}\to\N^{M_1}$ with $\pi(\tau)=\sigma$. 

Since $M_1$ sits over $\N^{M_1}$ in the way that it does and $\N^{M_1}$ is isomorphic to its image under $\pi$, we may build a copy $M_2$ of $M_1$ in such a way that extends the isomorphism $\pi$. $$\begin{tikzpicture}[scale=.5]
% M_1
\draw[fill=Yellow!40] (0,1) ellipse (2 and 4) node[above=1.3cm] {$M_1$};
\draw[fill=Gold] (0,0) ellipse (1 and 2) node {$\N^{M_1}$};
% M_2
\draw[densely dotted] (10,1) ellipse (2 and 4) node[above=1.3cm] {$M_2$};
\draw[fill=Gold] (10,0) ellipse (1 and 2) node {$\N^{M_1}$};
% the arrows
\draw[-Stealth] ($(0,0)+(30:1 and 2)$) to[out=20,in=160] node[below] {$\pi$} ($(10,0)+(150:1 and 2)$);
\draw[-Stealth,dotted] ($(0,1)+(55:2 and 4)$) to[out=20,in=160] node[above] {$\pi^*$} ($(10,1)+(125:2 and 4)$);
\end{tikzpicture}$$
That we may do so is an instance of the trivial general observation about structures that whenever you have made an isomorphic copy of part of a structure, then you may extend this to an isomorphic copy of the whole structure. In our case, we had an isomorphism $\pi$ of $\N^{M_1}$ with its copy (which happened to be $\N^{M_1}$ itself), and so we may extend this to an isomorphism $\pi^*$ of $M_1$ with a suitable structure $M_2$ extending $\N^{M_1}$. So we have a model $M_2$ which is a copy of $M_1$, witnessed by the isomorphism $\pi^*:M_1\to M_2$ extending $\pi$, in such a way that every element $m\in\N^{M_1}$ sits inside $M_1$ the same way that $\pi(m)$ sits inside~$M_2$.

Since $\pi$ was an automorphism of $\N^{M_1}$, the situation is therefore that $M_1$ and $M_2$ have exactly the same natural numbers $\N^{M_1}=\N^{M_2}$. Yet, $M_1$ thinks that the arithmetic sentence $\sigma$ is true in the natural numbers, while $M_2$ thinks $\sigma$ is false, because $M_1$ thinks $\tau$ is false, and so $M_2$ thinks $\pi(\tau)$ is false, but $\pi(\tau)=\sigma$.
\end{proof}

\section{Satisfacton is not absolute}\label{Section.SatisfactionIsNotAbsolute}

In this section we aim to show that the non-absoluteness phenomenon is pervasive. For any sufficiently rich structure $\mathcal{N}$ in any countable model of set theory $M$, there are elementary extensions $M_1$ and $M_2$, which have the structure $\mathcal{N}^{M_1}=\mathcal{N}^{M_2}$ in common, yet disagree about the satisfaction relation for this structure, in that $M_1$ thinks $\mathcal{N}\satisfies\sigma[\vec a]$ for some formula $\sigma$ and parameters $\vec a$, while $M_2$ thinks that $\mathcal{N}\satisfies\neg\sigma[\vec a]$ (see corollary \ref{Corollary.SameModelNDifferentTruths}). Even more generally, theorem \ref{Theorem.SameModelNDifferentS} shows that for any non-definable class $S$ in any structure $\mathcal{N}$ in any countable model of set theory $M$, there are elementary extensions $M_1$ and $M_2$ of $M$, which have the structure $\mathcal{N}^{M_1}=\mathcal{N}^{M_2}$ in common, yet disagree on the interpretation of the class $S^{M_1}\neq S^{M_2}$. We provide several curious instances as corollary applications.

\begin{theorem}\label{Theorem.SameModelNDifferentS}
Suppose that $M$ is a countable model of set theory, $\mathcal{N}$ is structure in $M$ in a finite language and $S\of\mathcal{N}$ is an additional predicate, $S\in M$, that is not definable in $\mathcal{N}$ with parameters. Then there are elementary extensions $M\elesub M_1$ and $M\elesub M_2$, which agree on the structure $\mathcal{N}^{M_1}=\mathcal{N}^{M_2}$, having all the same elements of it, the same language and the same interpretations for the functions, relations and constants in this language, yet disagree on the extension of the additional predicate $S^{M_1}\neq S^{M_2}$, even though by elementarity these predicates have all the properties in $M_1$ and $M_2$, respectively, that $S$ has in $M$.
\end{theorem}

\begin{proof}Fix such a structure $\mathcal{N}$ and non-definable class $S$ inside a countable model of set theory $M$. Let $M_1$ be any countable computably saturated elementary extension of $M$, in the language with constants for every element of $M$. In particular, the structure $\<\mathcal{N},S,m>_{m\in \mathcal{N}^M}^{M_1}$ is countable and computably saturated (and this is all we actually require). Since $S$ is not definable from parameters, there are objects $s$ and $t$ in $\mathcal{N}^{M_1}$ with the same $1$-type in $\<\mathcal{N},m>_{m\in\mathcal{N}^M}^{M_1}$, yet $s\in S^{M_1}$ and $t\notin S^{M_1}$. Since this latter structure is countable and computably saturated, it follows that there is an automorphism $\pi:\mathcal{N}\to\mathcal{N}$ with $\pi(t)=s$ and $\pi(m)=m$ for every $m\in \mathcal{N}^M$. That is, $\pi$ is an automorphism of $\mathcal{N}^{M_1}$ mapping $t$ to $s$, and respecting the copy of $\mathcal{N}^M$ inside $\mathcal{N}^{M_1}$. Let $M_2$ be a copy of $M_1$ containing $M$, witnessed by an isomorphism $\pi^*:M_1\to M_2$ extending $\pi$ and fixing the elements of $M$. It follows that $\mathcal{N}^{M_1}=\mathcal{N}^{M_2}$ and $M_1$ thinks $s\in S$, but since $M_1$ thinks $t\notin S$, it follows that $M_2$ thinks $\pi(t)\notin S$ and so $M_2$ thinks $s\notin S$, as desired.
\end{proof}

\begin{corollary}\label{Corollary.SameModelNDifferentTruths}
 Suppose that $M$ is a countable model of set theory, and that $\mathcal{N}$ is a sufficiently robust structure in $M$, in a finite language. Then there are elementary extensions $M\elesub M_1$ and $M\elesub M_2$, which agree on the natural numbers $\N^{M_1}=\N^{M_2}$ and on the structure $\mathcal{N}^{M_1}=\mathcal{N}^{M_2}$, having all the same elements of it, the same language and the same interpretations for the functions, relations and constants in this language, yet they disagree on what they each think is the standard satisfaction relation $\mathcal{N}\satisfies\sigma[\vec a]$ for this structure.
$$
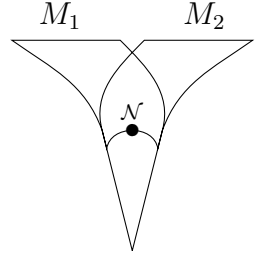
\begin{tikzpicture}[scale=.8]
 \node at (0,2) (N) {$\bullet$};
 \node[anchor=south] at (N) {\tiny$\mathcal{N}$};
 \draw (0,0) --(-.5,2) to [out=110, in=-35] (-2,3.5) --(-.2,3.5) to [out=-45,in=75] (.5,2) --(.425,1.7) to [out=100,in=0] (0,2) to [out=180,in=80] (-.425,1.7) --(-.5,2) to [out=105,in=-135] (.2,3.5) --(2,3.5) to [out=-145,in=70] (.5,2) --(0,0);
 \node[anchor=south] at (-1.2,3.5) {$M_1$};
 \node[anchor=south] at (1.2,3.5) {$M_2$};
\end{tikzpicture}
\qquad
\vbox{\hbox{$M\prec \ M_1,M_2\satisfies\ZFC$}\medskip\hbox{$\N^{M_1}=\N^{M_2},\qquad\mathcal{N}^{M_1}=\mathcal{N}^{M_2}$}\medskip\hbox{there are $\sigma$ and $\vec a$ for which}\medskip\hbox{$M_1$ believes $\mathcal{N}\satisfies\sigma[\vec a]$}\medskip\hbox{$M_2$ believes $\mathcal{N}\satisfies\neg\sigma[\vec a]$}}
$$
\end{corollary}

\begin{proof}
By `sufficiently robust' we mean that $\mathcal{N}=\<N,\ldots>$ interprets the standard model of arithmetic $\N$, so that it can handle the \Godel\ coding of formulas, and also that it has a definable pairing function, so that it contains (\Godel\ codes for) all finite tuples of its elements. We are assuming that $\mathcal{N}$ is sufficiently robust from the perspective of the model $M$. It follows by Tarski's theorem on the non-definability of truth that the satisfaction relation of $M$ for $\mathcal{N}$---that is, the relation $\Sat(\varphi,\vec a)$, which holds just in case  $\mathcal{N}\satisfies\varphi[\vec a]$ from the perspective of $M$---is not definable in $\mathcal{N}$. And so the current theorem is a consequence of theorem \ref{Theorem.SameModelNDifferentS}, using that relation.
\end{proof}

Many of the various examples of non-absoluteness that we have mentioned in this article can now be seen as instances of corollary  \ref{Corollary.SameModelNDifferentTruths}, as in the following further corollary. In particular, statement (2) of corollary \ref{Corollary.SameStructuresDifferentTruths} is a strengthening of theorem \ref{Theorem.SameNDifferentTruths}, since we now get the models $M_1$ and $M_2$ as elementary extensions of any given countable model $M$.

\begin{corollary}\label{Corollary.SameStructuresDifferentTruths} Every countable model of set theory $M$ has elementary extensions $M\elesub M_1$ and $M\elesub M_2$, respectively in each case, which\dots
 \begin{enumerate}
 \item agree on their natural numbers with successor and order $\<\N,S,{\lt}>^{M_1}=\<\N,S,{\lt}>^{M_2}$, but which disagree on the even numbers, the prime numbers and the powers of two, so that $M_1$ thinks some $n$ is a large odd prime number, but $M_2$ thinks it is a large power of $2$.
$$\begin{tikzpicture}[scale=.7, xscale=.8]
 \node at (0,1) (N) {$\bullet$};
 \node[anchor=south,scale=.8] at (N) {$\N$};
 \draw (0,1) --(0,0) --(-.3,1) to [out=110, in=-25](-2,3) --(-.2,3) to [out=-35,in=70](.3,1) --(-.3,1) to [out=110,in=-145] (.2,3) --(2,3) to [out=-155,in=70](.3,1) --(0,0);
 \node[anchor=south] at (-1,3) {$M_1$};
 \node[anchor=south] at (1,3) {$M_2$};
\end{tikzpicture}
\qquad
\vbox{\hbox{$M_1,M_2\satisfies\ZFC$}\medskip\hbox{$\<\N,S,{\lt}>^{M_1}=\<\N,S,{\lt}>^{M_2}$}\medskip\hbox{$M_1$ believes $\N\satisfies n$ is an odd prime}\medskip\hbox{$M_2$ believes $\N\satisfies n=2^k$ for some $k$}}
$$
 \item agree on their natural numbers with successor, addition and order $\<\N,S,{+},{\lt}>^{M_1}=\<\N,S,{+},{\lt}>^{M_2}$, but which disagree on natural-number multiplication, so that $M_1$ thinks $a\cdot b=c$ for some particular natural numbers, but $M_2$ disagrees.
$$
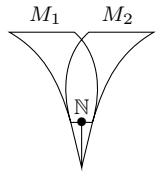
\begin{tikzpicture}[scale=.7, xscale=.8]
 \node at (0,1) (N) {$\bullet$};
 \node[anchor=south,scale=.8] at (N) {$\N$};
 \draw (0,1) --(0,0) --(-.3,1) to [out=110, in=-25](-2,3) --(-.2,3) to [out=-35,in=70](.3,1) --(-.3,1) to [out=110,in=-145] (.2,3) --(2,3) to [out=-155,in=70](.3,1) --(0,0);
 \node[anchor=south] at (-1,3) {$M_1$};
 \node[anchor=south] at (1,3) {$M_2$};
\end{tikzpicture}
\qquad
\vbox{\hbox{$M_1,M_2\satisfies\ZFC$}\medskip\hbox{$\<\N,S,{+},{\lt}>^{M_1}=\<\N,S,{+},{\lt}>^{M_2}$}\medskip\hbox{$M_1$ believes $\N\satisfies a\cdot b=c$}\medskip\hbox{$M_2$ believes $\N\satisfies a\cdot b\neq c$}}
$$
  \item agree on their standard model of arithmetic $\<\N,{+},{\cdot},0,1,{\lt}>^{M_1}=\<\N,{+},{\cdot},0,1,{\lt}>^{M_2}$, but which disagree on their theories of arithmetic truth.
$$
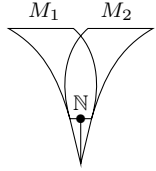
\begin{tikzpicture}[scale=.7, xscale=.8]
 \node at (0,1) (N) {$\bullet$};
 \node[anchor=south,scale=.8] at (N) {$\N$};
 \draw (0,1) --(0,0) --(-.3,1) to [out=110, in=-25](-2,3) --(-.2,3) to [out=-35,in=70](.3,1) --(-.3,1) to [out=110,in=-145] (.2,3) --(2,3) to [out=-155,in=70](.3,1) --(0,0);
 \node[anchor=south] at (-1,3) {$M_1$};
 \node[anchor=south] at (1,3) {$M_2$};
\end{tikzpicture}
\qquad
\vbox{\hbox{$M_1,M_2\satisfies\ZFC$}\medskip\hbox{$\N^{M_1}=\N^{M_2}$}\medskip\hbox{$M_1$ believes $\N\satisfies\sigma$}\medskip\hbox{$M_2$ believes $\N\satisfies\neg\sigma$}}
$$
  \item agree on their natural numbers $\N^{M_1}=\N^{M_2}$, their reals $\R^{M_1}=\R^{M_2}$ and their hereditarily countable sets $\<\HC,{\in}>^{M_1}=\<\HC,{\in}>^{M_2}$, but which disagree on their theories of projective truth.
$$\begin{tikzpicture}[scale=.7]
 \node at (0,2) (R) {$\bullet$};
 \node[anchor=south,scale=.8] at (R) {\tiny$\HC$};
 \node[anchor=west] at (.4,1.6) {\tiny$\omega$};
 \draw (0,2) --(0,0) --(-.5,2) to [out=110, in=-35] (-2,3.5) --(-.2,3.5) to [out=-45,in=75] (.5,2) --(.425,1.7) to [out=100,in=0] (0,2) to [out=180,in=80] (-.425,1.7) --(-.5,2) to [out=105,in=-135] (.2,3.5) --(2,3.5) to [out=-145,in=70] (.5,2) --(0,0);
 \draw (-.4,1.6) --(.4,1.6);
 \node[anchor=south] at (-1.2,3.5) {$M_1$};
 \node[anchor=south] at (1.2,3.5) {$M_2$};
\end{tikzpicture}
\qquad
\vbox{\hbox{$M_1,M_2\satisfies\ZFC$}\medskip\hbox{$\N^{M_1}=\N^{M_2}\qquad \R^{M_1}=\R^{M_2}$}\medskip\hbox{$\<\HC,{\in}>^{M_1}=\<\HC,{\in}>^{M_2}$}\medskip\hbox{$M_1$ believes $\HC\satisfies\sigma$}\medskip\hbox{$M_2$ believes $\HC\satisfies\neg\sigma$}}
$$
  \item agree on the structure $\<H_{\omega_2},{\in}>^{M_1}=\<H_{\omega_2},{\in}>^{M_2}$, but which disagree on truth in this structure.
$$
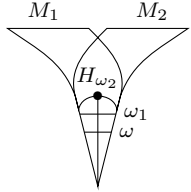
\begin{tikzpicture}[scale=.7]
 \node at (0,2) (R) {$\bullet$};
 \node[anchor=south,scale=.8] at (R) {\tiny$H_{\omega_2}$};
 \node[anchor=west] at (.4,1.6) {\tiny$\omega_1$};
 \node[anchor=west] at (.3,1.2) {\tiny$\omega$};
 \draw (0,2) --(0,0) --(-.5,2) to [out=110, in=-35] (-2,3.5) --(-.2,3.5) to [out=-45,in=75] (.5,2) --(.425,1.7) to [out=100,in=0] (0,2) to [out=180,in=80] (-.425,1.7) --(-.5,2) to [out=105,in=-135] (.2,3.5) --(2,3.5) to [out=-145,in=70] (.5,2) --(0,0);
 \draw (-.4,1.6) --(.4,1.6);
 \draw (-.3,1.2) --(.3,1.2);
 \node[anchor=south] at (-1.2,3.5) {$M_1$};
 \node[anchor=south] at (1.2,3.5) {$M_2$};
\end{tikzpicture}
\qquad
\vbox{\hbox{$M_1,M_2\satisfies\ZFC$}\medskip
 \hbox{$\<H_{\omega_2},{\in}>^{M_1}=\<H_{\omega_2},{\in}>^{M_2}$}\medskip
 \hbox{$M_1$ believes $H_{\omega_2}\satisfies\sigma$}\medskip
 \hbox{$M_2$ believes $H_{\omega_2}\satisfies\neg\sigma$}\medskip}
$$
  \item have a transitive rank-initial segment $\<V_\delta,{\in}>^{M_1}=\<V_\delta,{\in}>^{M_2}$ in common, but which disagree on truth in this structure.
$$\begin{tikzpicture}[scale=.7]
 \node at (0,2) (N) {$\bullet$};
 \node[anchor=south,scale=.8] at (N) {$V_\delta$};
 \draw (0,0) --(-.5,2) to [out=110, in=-35] (-2,3.5) --(-.2,3.5) to [out=-45,in=75] (.5,2) --(-.5,2) to [out=105,in=-135] (.2,3.5) --(2,3.5) to [out=-145,in=70] (.5,2) --(0,0);
 \node[anchor=south] at (-1.2,3.5) {$M_1$};
 \node[anchor=south] at (1.2,3.5) {$M_2$};
\end{tikzpicture}
\qquad
\vbox{\hbox{$M_1,M_2\satisfies\ZFC$}\medskip\hbox{$V_\delta^{M_1}=V_\delta^{M_2}\satisfies\ZFC$}\medskip\hbox{$M_1$ believes $V_\delta\satisfies\sigma$}\medskip\hbox{$M_2$ believes $V_\delta\satisfies\neg\sigma$}\medskip}
$$
 \end{enumerate}
\end{corollary}

\begin{proof}
What we mean by ``respectively'' is that each case may be exhibited separately, using different pairs of extensions $M_1$ and $M_2$. Each statement in the theorem is an immediate consequence of corollary \ref{Corollary.SameModelNDifferentTruths} or of theorem \ref{Theorem.SameModelNDifferentS}. Thanks to Roman Kossak for pointing out statement (2), which follows from theorem \ref{Theorem.SameModelNDifferentS} because multiplication is not definable in Presburger arithmetic $\<\N,+,\lt>$, as the latter is a decidable theory. Statement (1) is proved similarly, since the set of primes and the powers of two are not definable in $\<\N,S,\lt>$, since any nonstandard model of this structure consists of infinitely many $\Z$ chains above the standard part, and these admit numerous automorphisms by translation. By picking a suitable translation, we may move a power of two to an odd prime, and apply the argument in the proof of theorem \ref{Theorem.SameNDifferentTruths}.  For statement (3), we apply corollary \ref{Corollary.SameModelNDifferentTruths} to the structure $\mathcal{N}=\<\N,{+},{\cdot},0,1,{\lt}>^M$. For statement (4), one should clarify exactly what is meant by `projective truth,' since it can be viewed variously as the full second-order theory of the standard model of arithmetic, using $P(\N)^M$, or as the theory of the ordered real field $\<\R,\Z,{+},{\cdot},0,1,{\lt}>$ with a predicate for the integers, or as the theory of the set-theoretic structure $\<V_{\omega+1},{\in}>$, or of the structure of hereditarily countable sets $\HC=\<\HC,{\in}>$. Nevertheless, these models are all bi-interpretable with each other, and one can view projective truth as the satisfaction relation for any of them; in each case, we can make the conclusion of statement (4) by using that structure in corollary \ref{Corollary.SameModelNDifferentTruths}. Notice that one can similarly view arithmetic truth as residing in the structure of hereditarily finite sets $\<\HF,{\in}>$, which is mutually interpretable with the standard model of arithmetic via the Ackermann encoding, and in this case, statements (3), (4) and (5) can be seen as a progression, concerning the structures $H_\omega$, $H_{\omega_1}$ and $H_{\omega_2}$, a progression which continues, of course, to higher orders in set theory. Statement (6) is similarly an immediate consequence of corollary \ref{Corollary.SameModelNDifferentTruths}, using the structure $\mathcal{N}=\<V_\delta,{\in}>^M$.
\end{proof}

Let us explore a bit further how indefiniteness arises in the iterated truth-about-truth hierarchy. Beginning with the standard model of arithmetic $\N_0=\<\N,{+},{\cdot},0,1,{\lt}>$, we may define the standard truth predicate $\Tr_0$ for assertions in the language of arithmetic, and consider the structure $\N_1=\<\N,{+},{\cdot},0,1,{\lt},\Tr_0>$, in the expanded language with a predicate for the truth of arithmetic assertions. In this expanded language, we may make assertions both about arithmetic and about arithmetic truth. Climbing atop this structure, let $\Tr_1=\Th(\N_1)$ be its theory and form the next level of the iterated truths-about-truth hierarchy $\N_2=\<\N,{+},{\cdot},0,1,{\lt},\Tr_0,\Tr_1>$ by appending this new truths-about-truth predicate. We may easily continue in this way, building the finite levels of the hierarchy, each new truth predicate telling us about the truth of arithmetic assertions involving only previous levels of the iterated truth hierarchy. There is a rich literature on various aspects of this iterated truth hierarchy, from Tarski \cite{Tarski1935:TheConceptOfTruthInFormalizedLanguages} to Kripke  \cite{Kripke1975:OutlineOfATheoryOfTruth} and many others, including the related development of the revision theory of truth \cite{GuptaBelnap1993:TheRevisionTheoryOfTruth}. Feferman \cite{Feferman1991:ReflectingOnIncompleteness} treats the iteration of truth as an example where once one accepts certain statements about $\<\N,{+},{\cdot},0,1,{<},\Tr_0,\ldots,\Tr_n>$, then one ought accept certain other statements about $\<\N,{+},{\cdot},0,1,{<},\Tr_0,\ldots,\Tr_{n+1}>$.

\begin{corollary}\label{Corollary.IndefiniteIteratedTruth}
 For every countable model of set theory $M$ and any natural number $n$, there are elementary extensions $M_1$ and $M_2$ of $M$, which have the same natural numbers $\N^{M_1}=\N^{M_2}$, the same iterated arithmetic truth predicates $(\Tr_k)^{M_1}=(\Tr_k)^{M_2}$ for $k<n$ and hence the same iterated truth structure up to $n$,
$$\<\N,{+},{\cdot},0,1,{\lt},\Tr_0,\ldots,\Tr_{n-1}>^{M_1}=\<\N,{+},{\cdot},0,1,{\lt},\Tr_0,\ldots,\Tr_{n-1}>^{M_2},$$
 but which disagree on  the theory of this structure, and hence disagree on the next order of truth, $(\Tr_n)^{M_1}\neq (\Tr_n)^{M_2}$.
\end{corollary}

\begin{proof}
In other words, we will have $(\N_n)^{M_1}=(\N_n)^{M_2}$, yet $(\Tr_n)^{M_1}\neq(\Tr_n)^{M_2}$. This is an immediate consequence of corollary  \ref{Corollary.SameModelNDifferentTruths}, using the model $\mathcal{N}=(\N_n)^M$.
\end{proof}

In particular, even the cases $n=0,1$ or $2$ are interesting. The case $n=0$ amounts to theorem \ref{Theorem.SameNDifferentTruths}, and the case $n=1$ shows that one can have models of set theory $M_1$ and $M_2$ which have the same standard model of arithmetic $\N^{M_1}=\N^{M_2}$ and the same arithmetic truth $\Tr_0^{M_1}=\Tr_0^{M_2}$, yet disagree in their theory of the theory of arithmetic truth $\Tr_1^{M_1}\neq \Tr_1^{M_2}$. Thus, even if one assumes a definite nature for the structure of arithmetic, and also for arithmetic truth, then still there is indefiniteness as to the nature of truths about arithmetic truth, and so on throughout the iterated truth-about-truth hierarchy; indefiniteness arises at any particular level.

The process of iterating the truth hierarchy of course continues transfinitely, for as long as we have some natural way of representing the ordinals inside $\N$, in order to undertake the \Godel\ coding of formulas in the expanded language and retain the truth predicates as subclasses of $\N$. If $\alpha$ is any computable ordinal, for example, then we have a representation of $\alpha$ inside $\N$ using a computable relation of order type $\alpha$, and we may develop a natural \Godel\ coding for ordinals up to $\alpha$ and formulas in the language $\set{{+},{\cdot},0,1,{\lt},\Tr_\xi}_{\xi<\alpha}$. If $\eta<\alpha$ and $\N_\eta=\<\N,{+},{\cdot},0,1,{\lt},\Tr_\xi>_{\xi<\eta}$ is defined up to $\eta$, then we form $\N_{\eta+1}$ by adding the $\eta^\th$ order truth predicate $\Tr_\eta$ for assertions in the language of $\N_\eta$, which can make reference to the simpler truth predicates $\Tr_\xi$ for $\xi<\eta$ using the \Godel\ coding established by the computability of $\alpha$. The higher levels of this truths-about-truth hierarchy provide truth predicates for assertions about lower-level truths-about-truth for arithmetic. We note that indefiniteness cannot arise at limit ordinal stages, since when $\lambda$ is a limit ordinal, then a sentence $\sigma$ is true at stage $\lambda$ just in case it is true at any stage after which all the truth predicates appearing in $\sigma$ have arisen. In other words, $\Tr_\lambda=\Union_{\xi\lt\lambda}\Tr_\xi$ for any limit ordinal $\lambda$. Nevertheless, we find it likely that there is version of corollary \ref{Corollary.IndefiniteIteratedTruth} revealing indefiniteness in the transfinite realm of the iterated truth hierarchy.

We conclude this section with some further applications of theorem \ref{Theorem.SameModelNDifferentS}, using non-definable predicates other than a satisfaction predicate.

\begin{corollary}\
 \begin{enumerate}
  \item Every countable model of set theory $M$ has elementary extensions $M\elesub M_1$ and $M\elesub M_2$, which agree on their standard model of arithmetic $\N^{M_1}=\N^{M_2}$ and have a computable linear order $\trianglelt$ on $\N$ in common, yet $M_1$ thinks $\<\N,{\trianglelt}>$ is a well-order and $M_2$ does not.
  \item Similarly, every such $M$ has such $M_1$ and $M_2$, which agree on their standard model of arithmetic $\N^{M_1}=\N^{M_2}$ and thus agree on the computational behavior of all programs, yet they disagree on Kleene's $\mathcal O$, with $\mathcal{O}^{M_1}\neq\mathcal{O}^{M_2}$.
  \item Every countable model of set theory $M$ which thinks $0^\sharp$ exists has elementary extensions $M_1$ and $M_2$, which agree on the ordinals up to any desired uncountable cardinal $\kappa\in M$, on the constructible universe $L_\kappa^{M_1}=L_\kappa^{M_2}$ up to $\kappa$ and on the facts that $\kappa$ is an uncountable cardinal and $0^\sharp$ exists, yet disagree on which ordinals below $\kappa$ are the Silver indiscernibles. Similarly, we may ensure that they disagree on $0^\sharp$, so that $(0^\sharp)^{M_1}\neq(0^\sharp)^{M_2}$.
 \end{enumerate}
\end{corollary}

\begin{proof}
These are each consequences of theorem \ref{Theorem.SameModelNDifferentS}. For statement (1), consider the structure $\mathcal{N}=\N^M$, with the predicate $S=\WO^M$, the set of indices of computably enumerable well-orderings on $\N$ in $M$. This is a $\Pi^1_1$-complete set of natural numbers, and hence not first-order definable in the structure $\mathcal{N}$ from the perspective of $M$. Thus, by theorem \ref{Theorem.SameModelNDifferentS}, we get models of set theory $M_1$ and $M_2$, elementarily extending $M$, which agree on $\N^{M_1}=\N^{M_2}$, yet disagree on $\WO^{M_1}\neq\WO^{M_2}$. So there is some c.e.~relation $\trianglelt$ in common, yet $M_1$ (we may assume) thinks it is a well-order and $M_2$ does not. Since every c.e.~order is isomorphic to a computable order, we may furthermore assume that $\trianglelt$ is computable, and the models will compute it with the same program, and both see that it is a linear order, as desired. Statement (2) follows as a consequence, since using the program computing $\trianglelt$, we may construct deviations between $\mathcal{O}^{M_1}$ and $\mathcal{O}^{M_2}$; or alternatively, statement (2) follows immediately from theorem \ref{Theorem.SameModelNDifferentS}, since $\mathcal{O}$ is not definable in $\N$. Statement (3) also follows from theorem \ref{Theorem.SameModelNDifferentS}, since the class of Silver indiscernibles below $\kappa$ is not definable in $\<L_\kappa,{\in}>$ in $M$, as from it we could define a truth predicate for that structure. So there must be extensions which agree on $\<L_\kappa,{\in}>^{M_1}=\<L_\kappa,{\in}>^{M_2}$ and on the fact that $\kappa$ is an uncountable cardinal and that $0^\sharp$ exists, yet disagree on the Silver indiscernibles below $\kappa$. Similarly, s $0^\sharp$ also is not definable in $\<L_\kappa,\in>$, we may make them disagree on $0^\sharp$.
\end{proof}

We expect that the reader will be able to construct many further instances of the phenomenon. Here is another example we found striking.

\begin{theorem}\label{Theorem.NonabsolutenessOfBeingArithmetic}
 Every countable model of set theory $M$ has elementary extensions $M_1$ and $M_2$, which agree on the structure of their standard natural numbers $\<\N,{+},{\cdot},0,1,{\lt}>^{M_1}=\<\N,{+},{\cdot},0,1,{\lt}>^{M_2}$, and which have a set $A\of\N$ in common, extensionally identical in $M_1$ and $M_2$, yet $M_1$ thinks $A$ is first-order definable in $\N$ and $M_2$ thinks it is not.
\end{theorem}

The proof relies on the following lemma, which was conjectured by the first author and asked on MathOverflow  \cite{MO141387Hamkins:IsThereASubsetOfTheNaturalNumberPlaneWhichDoesntKnowWhichOfItsSlicesAreArithmetic?} specifically in connection with this application. The question was answered there by Andrew Marks, whose proof we adapt here.

\begin{sublemma}[Andrew Marks \cite{MO141387Hamkins:IsThereASubsetOfTheNaturalNumberPlaneWhichDoesntKnowWhichOfItsSlicesAreArithmetic?}]\label{Lemma.SubsetOfPlane}
 There is a subset $B\of\N\times\N$, such that the set $\set{n\in\N\mid B_n\text{ is arithmetic}}$ is not definable in the structure $\<\N,{+},{\cdot},0,1,{\lt},B>$, where $B_n=\set{k\mid(n,k)\in B}$ denotes the $n^\th$ section of $B$.
\end{sublemma}

\begin{proof}
We identify $\N\times\N$ with $\N$ via \Godel\ pairing. We use $X'$ and $X^{(n)}$ to denote the Turing jump of $X$ and the $n^{\rm th}$ Turing jump of $X$, respectively. Recall that a set $X\of\N$ is {\df $n$-generic} if for every $\Sigma^0_n$ subset $S\of 2^\ltomega$, there is an initial segment of $X$ that either is in $S$ or has no extension in $S$. A set is {\df arithmetically generic} if it is $n$-generic for every $n$. It is a standard fact that if $X$ is $1$-generic, then $X' \equiv_T 0' \oplus X$, and if $X$ is $n$-generic and $Y$ is $1$-generic relative to $X \oplus 0^{(n-1)}$, then $X \oplus Y$ is $n$-generic. Let $A = 0^{(\omega)} = \oplus_{n} 0^{(n)}$, which is Turing equivalent to the set of true sentences of first order arithmetic. We shall construct a set $B$ with the following features:
\begin{enumerate}
  \item $\set{n \in \mathbb{N} \mid B_n\text{ is arithmetic}}=A$. More specifically, if $n \in A$, then $B_n$ is $(n+1)$-generic and computable from $0^{(n+1)}$, and if $n \notin A$, then $B_n$ is arithmetically generic.
 \item For each natural number $k$, the set $C_k = \oplus_{i \in \omega} B_{m_i}$ is\\ $(k+1)$-generic, where $m_i$ is the $i$th element of the set\\ $\set{m \in \N \mid m \notin A \text{ or }  k\leq m}$.
\end{enumerate}
Any set $B$ with these features, we claim, fulfills the lemma. To see this, we argue first by argue by induction that $B^{(n)} \equiv_T 0^{(n)} \oplus C_{n}$ for any natural number $n$. This is immediate for $n=0$, since $C_0=\oplus_m B_m$. If $B^{(n)} \equiv_T 0^{(n)} \oplus C_{n}$, then $B^{(n+1)} \equiv_T (0^{(n)} \oplus C_n)' \equiv_T 0^{(n+1)} \oplus C_n$, since $C_n$ is $(n+1)$-generic and hence $1$-generic relative to $0^{(n)}$, but $0^{(n+1)} \oplus C_n \equiv_T 0^{(n+1)} \oplus C_{n+1}$ because either $n \notin A$ and so $C_n = C_{n+1}$ or $n \in A$ so $C_{n+1} \equiv_T B_n \oplus C_n$, since $B_n \leq_T 0^{(n+1)}$. So we have established $B^{(n)}\equiv_T 0^{(n)}\oplus C_n$. Since $C_n$ is $(n+1)$-generic, it follows that $0^{(n)}\oplus C_n$ does not compute $0^{(n+1)}$, and so also $B^{(n)}$ does not compute $0^{(n+1)}$. In particular, $B^{(n)}$ does not compute $A$ for any $n$, and so $A$ is not arithmetically definable from $B$, as desired.

It remains to construct the set $B$ with features (1) and (2). We do so in stages, where after stage $n$ we will have completely specified $B_0, B_1, \ldots, B_n$ and finitely much additional information about $B$ on larger coordinates. To begin, let $B_0$ be any set satisfying the requirement of condition (1). We will ensure inductively that after each stage $n$, the set $C_{k,n} = B_{m_0} \oplus \ldots \oplus B_{m_j}$ is $(k+1)$-generic, where $k\leq n$ and $m_0, \ldots, m_j$ are the elements of $\set{m \in \N \mid m \notin A \text{ or }k\leq m} \cap \{0, \ldots, n\}$.

At stage $n > 0$, for each of the finitely many pairs $(i,k)$ with $i,k < n$, we let $S_{i,k}$ be the $i$th $\Sigma^0_{k+1}$ subset of $2^{<\omega}$, and if possible, we make a finite extension to our current approximation to $B$ so that the resulting approximation to $C_k$ extends an element of $S_{i,k}$, thereby ensuring this instance of (2). If there is no such extension, then since inductively $C_{k,n-1}$ is $(k+1)$-generic, there is already a finite part of our current approximation to $B$ that cannot be extended to extend an element of $S_{i,k}$, and this also ensures this instance of (2).

We complete stage $n$ by specifying $B_n$. If $n \notin A$, then we simply extend our current approximation to $B$ by ensuring that $B_n$ is arithmetically generic relative to $B_0 \oplus \ldots \oplus B_{n-1}$. This ensures this instance of (1) while maintaining our induction assumption that $C_{k,n}$ is $(k+1)$-generic for each $k < n$, since $C_{k,n-1}$ is $(k+1)$-generic and $B_n$ is $(k+1)$-generic relative to it; and similarly, $C_{n,n}$ is now $(n+1)$-generic. If $n \in A$, then we let $B_n$ be any $0^{(n+1)}$-computable $(n+1)$-generic set extending the finitely many bits of $B_n$ specified in the current approximation. For each $k < n$, let $j_0, \ldots, j_t$ be the elements of $A$ in the interval $[k,n)$. Since our new $B_n$ is $1$-generic relative to $0^{(n)}$, which can compute $B_{j_0} \oplus \ldots \oplus B_{j_t}$, it follows that that $B_{j_0} \oplus \ldots \oplus B_{j_t} \oplus B_n$ is $(k+1)$-generic, and so $C_{k,n}$ is $(k+1)$-generic, as the remaining elements in the finite join defining $C_{k,n}$ are mutually arithmetically generic with this; and since $C_{n,n}$ is $n+1$-generic, we maintain our induction assumption.

This completes the construction. We have fulfilled (1) explicitly by the choice of $B_n$, and we fulfilled (2) by systematically deciding all the required sets $S_{i,k}$.
\end{proof}

\begin{proof}[Proof of theorem \ref{Theorem.NonabsolutenessOfBeingArithmetic}] Fix any countable model of set theory $M$. Apply lemma \ref{Lemma.SubsetOfPlane} inside $M$, to find a predicate $B\of\N\times\N$ in $M$, such that the set $S=\set{n\in\N\mid B_n\text{ is arithmetic}}$ is not definable in the structure $\mathcal{N}=\<\N,{+},{\cdot},0,1,{\lt},B>$ from the perspective of $M$. It follows by theorem \ref{Theorem.SameModelNDifferentS} applied to this structure that there are elementary extensions $M_1$ and $M_2$ of $M$, which agree on $\mathcal{N}^{M_1}=\mathcal{N}^{M_2}$ and in particular on $\N^{M_1}=\N^{M_2}$ and $B^{M_1}=B^{M_2}$, but not on $S^{M_1}\neq S^{M_2}$. In particular, there is some section $A=B_n$ that is arithmetic in $M_1$ (we may assume), but not in $M_2$. But since the models agree on the predicate $B$, they agree on all the sections of $B$ and in particular have the set $A$ extensionally in common.
\end{proof}

\section{Indefiniteness for specific types of sentences}\label{Section.SpecificSentences}

Earlier in this article, we proved that the satisfaction relation $\mathcal{N}\satisfies\sigma$ for a first-order structure $\mathcal{N}$ is not generally absolute between the various models of set theory containing that model and able to express this satisfaction relation. But the proofs of non-absoluteness did not generally reveal any specific nature for the sentences on which truth can differ in different models of set theory. We should now like to address this issue by presenting an alternative elementary proof of non-absoluteness, using reflection and compactness, which shows that the theory of a structure can vary on sentences whose specific nature we can identify.

A cardinal $\delta$ is defined to be {\df $\Sigma_n$-correct}, if $V_\delta\elesub_{\Sigma_n} V$. The reflection theorem shows that there is a proper class club $C^{(n)}$ of such cardinals. The cardinal $\delta$ is {\df fully correct}, if it is $\Sigma_n$-correct for every $n$. This latter notion is not expressible as a single assertion in the first-order language of \ZFC, but one may express it as a scheme of assertions about $\delta$, in a language with a constant for $\delta$. Namely, let ``$V_\delta\elesub V$'' denote the theory asserting of every formula $\varphi$ in the language of set theory, that $\forall x\in V_\delta[\varphi(x)\iff \varphi(x)^{V_\delta}]$. Since every finite subtheory of this scheme is proved consistent in \ZFC\ by the reflection theorem, it follows that the theory is finitely consistent and so, by the compactness theorem, $\ZFC+{V_\delta\elesub V}$ is equiconsistent with \ZFC. The reader may note that since $V_\delta\elesub V$ asserts the elementarity separately for each formula, we may not deduce in this theory that $V_\delta\satisfies\ZFC$, but rather only for each axiom of \ZFC\ separately, that $V_\delta$ satisfies that axiom.

\begin{theorem}\label{Theorem.FailureOfCollectionEvenOdd}
 Every countable model of set theory $M\satisfies\ZFC$ has elementary extensions $M_1$ and $M_2$, with a transitive rank-initial segment  $\<V_\delta,{\in}>^{M_1}=\<V_\delta,{\in}>^{M_2}$ in common, such that $M_1$ thinks that the least natural number $n$ for which $V_\delta$ violates $\Sigma_n$-collection is even, but $M_2$ thinks it is odd.
$$
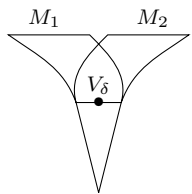
\begin{tikzpicture}[scale=.6]
 \node at (0,2) (N) {$\bullet$};
 \node[anchor=south] at (N) {$V_\delta$};
 \draw (0,0) --(-.5,2) to [out=110, in=-35] (-2,3.5) --(-.2,3.5) to [out=-45,in=75] (.5,2) --(-.5,2) to [out=105,in=-135] (.2,3.5) --(2,3.5) to [out=-145,in=70] (.5,2) --(0,0);
 \node[anchor=south] at (-1.2,3.5) {$M_1$};
 \node[anchor=south] at (1.2,3.5) {$M_2$};
\end{tikzpicture}
\qquad
\vbox{\hbox{$M_1,M_2\satisfies\ZFC$}\medskip\hbox{$V_\delta^{M_1}=V_\delta^{M_2}$}\medskip\hbox{$n$ is least with $\neg\Sigma_n$-collection in $V_\delta$}\medskip\hbox{$M_1$ believes $n$ is even}\medskip\hbox{$M_2$ believes $n$ is odd}}
$$
\end{theorem}

\begin{proof}
Suppose that $M$ is a countable model of \ZFC, and consider the theory:
\begin{multline*}
  T_1=\Delta(M) + \ {V_\delta \elesub V}\  + \ \set{m\in V_\delta \mid m\in M}\\
+ \ \text{the least $n$ such that $V_\delta\nvDash\Sigma_n$-collection is even,}
\end{multline*}
where $\Delta(M)$ is the elementary diagram of $M$, in the language of set theory having constants for every element of $M$. Note that the first three components of $T_1$ as it is described above are each infinite schemes, whereas the final assertion ``the least $n$\ldots,'' which we take also to assert that there is such an $n$, is expressible as a single sentence about $\delta$ in the language of set theory, using the \ZFC-definable satisfaction relation for $\<V_\delta,{\in}>$. We claim that this theory is consistent. Consider any finite subtheory $t\of T_1$. We shall find a $\delta$ in $M$ such that $M$ with this $\delta$ will satisfy every assertion in $t$. Let $k$ be a sufficiently large odd number, so that every formula appearing in any part of $t$ has complexity at most $\Sigma_k$, and let $\delta$ be the next $\Sigma_k$-correct cardinal in $M$ above the largest rank of an element of $M$ whose constant appears in $t$. We claim now that $\<M,{\in^M},\delta,m>_{m\in M}\satisfies t$. First, it clearly satisfies all of $\Delta(M)$; and since $\delta$ is $\Sigma_k$-correct in $M$, we have $V_\delta^M\elesub_{\Sigma_k} M$, and since also $\delta$ was large enough to be above any of the constants of $M$ appearing in $t$, we attain any instances from the first three components of the theory that are in the finite sub-theory $t$; since $\delta$ is $\Sigma_k$-correct in $M$, it follows that $V_\delta^M$ satisfies every instance of $\Sigma_k$-collection; but since $\delta$ is not a limit of $\Sigma_k$-correct cardinals (since it is the ``next'' one after a certain ordinal), it follows that $V_\delta^M$ does not satisfy $\Sigma_{k+1}$-collection, and so the least $n$ such that $V_\delta^M\not\satisfies\Sigma_n$-collection is precisely $n=k+1$, which is even. So $T_1$ is finitely consistent and thus consistent. Similarly, the theory
\begin{multline*}
T_2=\Delta(M)+ \ {V_\delta \elesub V} \ + \ \set{m\in V_\delta\mid m\in M}\\
   +\text{the least $n$ such that $V_\delta\nvDash\Sigma_n$-collection is odd}
\end{multline*}
is also consistent.

Let $\<M_1,M_2>$ be a computably saturated model pair, such that $M_1\satisfies T_1$ and $M_2\satisfies T_2$. It follows that $\<V_\delta^{M_1},V_\delta^{M_2}>$ is a computably saturated model pair of elementary extensions of $\<M,{\in^M}>$, which are therefore elementarily equivalent in the language of set theory with constants for elements of $M$, and hence isomorphic by an isomorphism respecting those constants. So we may assume without loss of generality that $\<M,{\in^M}>\prec \<V_\delta,{\in}>^{M_1}=\<V_\delta,{\in}>^{M_2}$. Meanwhile, $M_1$ thinks that this $V_\delta$ violates $\Sigma_n$-collection first at an even $n$ and $M_2$ thinks it does so first for an odd $n$, since these assertions are part of the theories $T_1$ and $T_2$, respectively.
\end{proof}

An alternative version of theorem \ref{Theorem.FailureOfCollectionEvenOdd}, with essentially the same proof, produces elementary extensions $M_1$ and $M_2$ of $M$ with a rank initial segment $V_\delta^{M_1}=V_\delta^{M_2}$ in common, but $M_1$ thinks the least $n$ for which $V_\delta$ is not $\Sigma_n$-correct is even, but $M_2$ thinks it is odd.

By looking not just at the parity of the least $n$ where $\Sigma_n$-collection (or $\Sigma_n$-correctness) fails, but rather, say, at the $k^\th$ binary digit, we can easily make infinitely many different elementary extensions $M_1,M_2,\ldots$ of $M$, with the natural numbers $\N^{M_k}$ and $V_\delta^{M_k}$ all in common, but such that $M_k$ thinks that the least $n$ for which this $V_\delta$ violates $\Sigma_n$-collection is a number with exactly $k$ many prime factors. In particular, even though they have the same structure $\<V_\delta,{\in}>^{M_k}$, they each think specific incompatible things about the theory of this structure.

\section{``Being a model of \ZFC'' is not absolute}\label{Secton.BeingModelZFCNotAbsolute}

In this section, we prove that the question of whether a given transitive rank initial segment $V_\delta$ of the universe is a model of \ZFC\ is not absolute between models of set theory with that rank initial segment in common. Recall that a cardinal $\delta$ is {\df worldly}, if $V_\delta\satisfies\ZFC$.

\begin{theorem}\label{Theorem.SatisfyingZFCNotAbsolute}
 If $M$ is a countable model of set theory in which the worldly cardinals form a stationary proper class (it would suffice, for example, that $M\satisfies\Ord$ is inaccessible), then there are elementary extensions $M_1$ and $M_2$, which have a transitive rank initial segment $\<V_\delta,{\in}>^{M_1}=\<V_\delta,{\in}>^{M_2}$ in common, such that $M_1$ thinks $V_\delta\satisfies\ZFC$ but $M_2$ thinks $V_\delta\not\satisfies\ZFC$. Moreover, such extensions can be found for which $\delta$ is fully correct in both $M_1$ and $M_2$, and furthermore in which $M\elesub V_\delta^{M_1}=V_\delta^{M_2}$.
$$
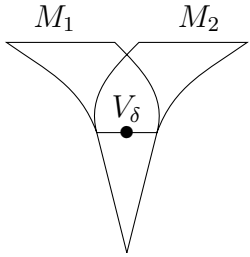
\begin{tikzpicture}[scale=.8]
 \node at (0,2) (N) {$\bullet$};
 \node[anchor=south] at (N) {$V_\delta$};
 \draw (0,0) --(-.5,2) to [out=110, in=-35] (-2,3.5) --(-.2,3.5) to [out=-45,in=75] (.5,2) --(-.5,2) to [out=105,in=-135] (.2,3.5) --(2,3.5) to [out=-145,in=70] (.5,2) --(0,0);
 \node[anchor=south] at (-1.2,3.5) {$M_1$};
 \node[anchor=south] at (1.2,3.5) {$M_2$};
\end{tikzpicture}
\qquad
\vbox{\hbox{$M_1,M_2\satisfies\ZFC$}\medskip\hbox{$\<V_\delta,{\in}>^{M_1}=\<V_\delta,{\in}>^{M_2}$}\medskip\hbox{$M_1$ believes $V_\delta\satisfies\ZFC$}\medskip\hbox{$M_2$ believes $V_\delta\not\satisfies\ZFC$}}
$$
\end{theorem}

\begin{proof}
Fix any countable model $M\satisfies\ZFC$, such that the worldly cardinals form a stationary proper class in $M$. That is, every definable proper class club $C\of\Ord$ in $M$ contains some $\delta$ that is worldly in $M$. Let $T_1$ be the theory consisting of the elementary diagram $\Delta(M)$ plus the scheme of assertions $V_\delta\elesub V$, in the language with a new constant symbol for $\delta$, plus the assertions $a\in V_\delta$ for each constant symbol $a$ for an element $a\in M$, plus the assertion ``$\delta$ is worldly,'' which is to say, the assertion that $V_\delta\satisfies\ZFC$. Suppose that $t$ is a finite subtheory of $T_1$, which therefore involves only finitely many instances from the $V_\delta\elesub V$ scheme. Let $n$ be large enough so that all the formulas $\varphi$ in these instances have complexity at most $\Sigma_n$. Since the $\Sigma_n$-correct cardinals form a closed unbounded class and the worldly cardinals are stationary, there is a $\Sigma_n$-correct worldly cardinal $\delta$ in $M$. It follows that $M$ satisfies all the formulas in $t$ using this $\delta$, and so the theory $T_1$ is finitely consistent and hence consistent.

Let $T_2$ be the theory consisting of the elementary diagram $\Delta(M)$, the scheme $V_\delta\elesub V$ plus the assertion ``$\delta$ is not worldly.'' This theory also is finitely consistent, since if $t\of T_2$ is finite, then let $n$ be beyond the complexity of any formula appearing as an instance of $V_\delta\elesub V$ in $t$, and let $\delta$ be a $\Sigma_n$ correct cardinal in $M$ that is not worldly (for example, we could let $\delta$ be the next $\Sigma_n$-correct cardinal after some ordinal; this can never be worldly since $V_\delta$ will not satisfy $\Sigma_n$-reflection). It follows that $M$ with this $\delta$ satisfies every assertion in $t$, showing that $T_2$ is finitely consistent and hence consistent.

$$\begin{tikzpicture}[yscale=.7]
\draw (0,0) node {$\<M,\in^M>$};
\draw (5,1) node {$\<V_\delta,\in>^{M_1}\elesub \<M_1,\in^{M_1}>$};
\draw (5,-1) node {$\<V_\delta,\in>^{M_2}\elesub \<M_2,\in^{M_2}>$};
\draw (1.85,.5) node[rotate=20] {$\prec$};
\draw (1.85,-.5) node[rotate=-20] {$\prec$};
\draw (3.5,0) node[rotate=90] {$=$};
\end{tikzpicture}$$

Let $\<M_1,M_2>$ be a computably saturated model pair, where $M_1\satisfies T_1$ and $M_2\satisfies T_2$. It follows that $\<V_\delta^{M_1},V_\delta^{M_2}>$ is also a computably saturated model pair of models of set theory, and these both satisfy the elementary diagram of $M$. Consequently, they are isomorphic by an isomorphism that respects the interpretation of $M$ in them, and so by replacing with an isomorphic copy, we may assume that $\<V_\delta,{\in}>^{M_1}=\<V_\delta,{\in}>^{M_2}$. The theories $T_1$ and $T_2$ ensure that $V_\delta\elesub V$ in both $M_1$ and $M_2$, that $M\elesub V_\delta$, and furthermore, that $M_1\satisfies\delta$ {\it is worldly} and $M_2\satisfies\delta$ {\it is not worldly}, or in other words, $M_1\satisfies (V_\delta\satisfies\ZFC)$, but $M_2\satisfies (V_\delta\not\satisfies\ZFC)$, as desired.
\end{proof}

The hypothesis that the worldly cardinals form a stationary proper class is a consequence of the (strictly stronger) \Levy\ scheme, also known as {\it $\Ord$ is Mahlo}, asserting in effect that the inaccessible cardinals form a stationary proper class. This is in turn strictly weaker in consistency strength than the existence of a single Mahlo cardinal, since if $\kappa$ is Mahlo, then $V_\kappa\satisfies\Ord$ {\it is Mahlo}. So these are all rather weak large cardinal hypotheses, in terms of the large cardinal hierarchy. Meanwhile, the conclusion already explicitly has large cardinal strength, since $M_1\satisfies\delta$ is worldly. Furthermore, the ``Moreover,\ldots'' part of the conclusion makes the hypothesis optimal, since if $M_1\satisfies\delta$ is worldly and fully correct, then the worldly cardinals of $M_1$ form a stationary proper class in $V_\delta^{M_1}$, as any definable class club there extends to a class club in $M_1$ containing $\delta$.

As an example to illustrate the range of non-absoluteness, consider the case where $M$ is a model of \ZFC\ in which the worldly limits of Woodin cardinals are stationary. It follows from theorem \ref{Theorem.SatisfyingZFCNotAbsolute} that there are elementary extensions $M_1$ and $M_2$ of $M$ with a common rank initial segment $\<V_\delta,{\in}>^{M_1}=\<V_\delta,{\in}>^{M_2}$, such that
 $$M_1\satisfies(V_\delta\satisfies\ZFC+\text{there are a proper class of Woodin cardinals}),$$
but
 $$M_2\satisfies(V_\delta\not\satisfies\ZFC+\text{there are a proper class of Woodin cardinals}).$$
The proof shows that from the perspective of $M_2$, the common $V_\delta$ does not satisfy $\Sigma_n$-replacement for sufficiently large $n$, even though it satisfies full replacement from the perspective of $M_1$, which has all the same formulas.

We conclude this section with an observation that many have found curious. We learned this result from Brice Halimi \cite{Halimi:ModelsAndUniverses}, who emphasized that the proof is classical and the result may be folklore. If $\mathcal{M}=\<M,\in^M>$ is any model of set theory, and $\mathcal{M}\satisfies\<m,E>$ is a first-order structure, then we may regard $\<m,E>$ as an actual first-order structure in $V$, with domain $\set{a\in M\mid\mathcal{M}\satisfies a\in m}$ and relation $a\mathrel E b\iff \mathcal{M}\satisfies a\mathrel{E} b$. Thus, we extract the existential content of the structure from $\mathcal{M}$ to form an actual structure in $V$.

\begin{theorem}[\cite{Halimi:ModelsAndUniverses}]\label{Theorem.EveryModelHasAModel}
 Every model of \ZFC\ has an element that is a model of \ZFC. Specifically, if $\<M,\in^M>\satisfies\ZFC$, then there is an object $\<m,E>$ in $\mathcal{M}$, which when extracted as an actual structure in $V$, satisfies \ZFC.
\end{theorem}

\begin{proof}
Consider first the case that $\mathcal{M}$ is $\omega$-nonstandard. By the reflection theorem, any particular finite fragment of \ZFC\ is true in some $\<V_\delta^M,{\in}>^{\mathcal{M}}$, and so by overspill there must be some $\delta$ in $\mathcal{M}$ for which $\mathcal{M}$ believes $\<V_\delta,{\in}>^{\mathcal{M}}$ satisfies a nonstandard fragment of $\ZFC$, and so in particular it will satisfy the actual \ZFC, as desired. Meanwhile, if $\mathcal{M}$ is $\omega$-standard, then it must satisfy $\Con(\ZFC)$ and so it can build the Henkin model of \ZFC. So in any case, there is a model of \ZFC\ inside $\mathcal{M}$.
\end{proof}

What surprises some is that the argument succeeds even when $\mathcal{M}\satisfies\neg\Con(\ZFC)$, for although $\mathcal{M}$ thinks there is no model of \ZFC, nevertheless it has many actual models of \ZFC, which it rejects because its $\omega$-nonstandard nature causes it to have a false understanding of what \ZFC\ is. In the $\omega$-nonstandard case of theorem \ref{Theorem.EveryModelHasAModel}, the model $\<V_\delta,{\in}>^{\mathcal{M}}$ that is produced is actually a model of \ZFC, although $\mathcal{M}$ thinks it fails to satisfy some nonstandard part of \ZFC. Perhaps one way of highlighting the issue is to point out that the theorem asserts that every model $\mathcal{M}$ of \ZFC\ has an element $\<m,E>$, such that for every axiom $\sigma$ of \ZFC, the model $\mathcal{M}$ believes $\<m,E>\satisfies\sigma$; this is different from $\mathcal{M}$ believing that $\<m,E>$ satisfies every axiom of \ZFC, since when $\mathcal{M}$ is $\omega$-nonstandard it has additional axioms not present in the metatheory.

\section{Conclusions}\label{Section.Conclusions}

Let us now return to the motivating philosophical issue with which we began this article. In contention is whether the definiteness of arithmetic truth can be seen as a consequence of the definiteness of the underlying mathematical objects and structure, or to put it simply, whether definiteness-about-truth follows from definiteness-about-objects. Feferman implicitly supports this in his conclusion:
\begin{quote}
 In my view, the conception [of the bare order structure of the natural numbers $\N$, with its least element and the attendant operations of successor and predecessor] is {\it completely clear}, and thence {\it all arithmetical statements are definite}. \cite[p.6--7]{Feferman:CommentsForEFIWorkshop} (emphasis original)
\end{quote}
Donald Martin is more explicit in his article for the same conference series:
\begin{quote}
The concept of the natural numbers is first-order complete: it determines
truth values for all sentences of first-order arithmetic. That is, it implies each
first-order sentence or its negation. (p. 3)

\smallskip\noindent
What I am suggesting is that the real reason for confidence in first-order completeness is our confidence in the full determinateness of the concept of the natural numbers. \cite[p. 13]{Martin:CompletenessOrIncompletenessOfBasicMathematicalConcepts}
\end{quote}

Let us give a name to the general principle underlying these remarks, in order to help focus the discussion:
\begin{quote}
 {\bf Objects$\to$Truth determinateness principle}. Necessarily, if the objects and underlying operations and relations of a mathematical structure are definite and determined, then the corresponding theory of truth for that structure also is definite and determined.
\end{quote}
According to this principle, there can be no indeterminateness in our theory of arithmetic, if the mathematical structure in which that truth resides is definite and determined. No blurriness or fuzziness will sneak into our judgment of truth in a structure, if there is none in the underlying mathematical objects and structure. In short, the view is that whenever you know exactly what a structure is, then truth in that structure is fully determined.

Our main philosophical thesis is that we don't necessarily get determinate truth in this way, for free, from the determinateness of the underlying objects, but rather, determinate truth constitutes an additional higher-order commitment, which must be argued for and justified separately.
In particular, we do not legitimately deduce the determinateness of arithmetic truth based solely on our determinate understanding of the numbers $0$, $1$, $2$, $3$, and so on and their arithmetic structure.

%
%Weak: What does it mean to say that a mathematical structure is definite or determinate? This is, of course, a difficult philosophical problem of its own, on which Feferman and others have written at length \cite{?}\margin{references?}. For our purpose here, since we seek to focus detailed attention on a single step of reasoning, the step from definiteness-of-objects to definiteness-of-truth, let us simply understand these terms in the sense that they are used by Feferman and Martin.

\subsection{What is definiteness?}

What does it mean exactly to say that mathematical objects and mathematical structure are definite or that truth in a structure is definite or determined? We are inclined to place the burden of providing a full account of definiteness on those who are making the positive claim, that is, on our opponents who are supporting (if implicitly) the Objects$\to$Truth determinateness principle.

Meanwhile, however, in order to support our criticism of that too-easy jump from objects to truth, let us mention a few general ideas that we would find reasonable for any such notion of definiteness on offer to exhibit. To begin, in our view the notions of definiteness and determinateness of objects, structure, and truth are inherently concerned with a comparison of the objects and structure in different models or universe concepts and with the evaluation of truth judgments about that structure in those various worlds. That is, definiteness at bottom is an absoluteness claim made about a structure with respect to a class of worlds or universe concepts, asserting that the relevant objects, structure, and truth judgments are invariant across those worlds.

In this light, we would find it to be a sufficient condition for the definiteness of objects in a mathematical structure amongst a collection of worlds for those objects and that structure to exist identically in those worlds---all the worlds in question agree on which objects there are in the given structure and furthermore they are in perfect agreement about the fundamental atomic features of that structure, such as the interpretations of relations and operations in a first-order language. In such a case of agreement, we hold, the objects and structure could be said to be definite across those various worlds.

Similarly, we would find it to be a necessary condition for the definiteness of a given truth assertion about such a definite structure across a collection of worlds that they all come to the same judgement about the truth value of that assertion. That is, if the truth assertion is definite with respect to a collection of worlds or universe concepts, then at the very least, the assertion must have the same truth value in all those worlds. 

\subsection{Criticism of Objects$\to$Truth}

In light of these considerations, let us now argue for our central thesis, which is that the definiteness of the theory of truth for a structure does not necessarily follow as a consequence of the definiteness of the structure in which that truth resides.
%, but rather amounts to a higher-order ontological commitment requiring its own, separate justification.
In particular, we deny the Objects$\to$Truth determinateness principle. 

Even in the case of arithmetic truth and the standard model of arithmetic $\N$, we claim, it is a philosophical error to conclude that arithmetic truth is definite just on the basis that the natural numbers themselves and their structure $\<\N,{+},{\cdot},0,1,{\lt}>$ are definite. Definiteness of truth does not ride for free upon definiteness of objects, but rather one must say something more about it.

Our argument is a simple appeal to the mathematical results of the previous sections of this article, particularly theorems \ref{Theorem.SameNDifferentTruths} and \ref{Theorem.SameModelNDifferentS} and their corollaries, which show how it can be that we have definiteness about the objects and structural relations of a certain structure across a collection of worlds, while there is no definiteness for the truths residing in that structure across those worlds.

Namely, in the context of theorem \ref{Theorem.SameNDifferentTruths}, we have two set-theoretic worlds, $M_1$ and $M_2$, which agree completely on the objects and atomic structure of what they both agree is the standard model of arithmetic $\<\N,+,\cdot,0,1,<>^{M_1}=\<\N,+,\cdot,0,1,<>^{M_2}$, and yet they disagree on their judgments of arithmetic truth in this structure. They have a sentence $\sigma$ in common, which they both agree makes a certain arithmetic assertion---they both agree on the \Godel\ code of the sentence and on the syntactic structure of this sentence---and yet $M_1$ thinks that $\sigma$ is true in $\<\N,+,\cdot,0,1,<>$ while $M_2$ thinks that $\sigma$ is not true there. 

Because the two models $M_1$ and $M_2$ agree exactly on the objects and primitive structural relations of this structure, they fulfill the sufficient condition for the definiteness of objects for this structure. And indeed we could consider these models in the context of the larger collection of worlds that agree with them on the structure $\<\N,+,\cdot,0,1,<>$. This structure exists identically in all these worlds, and so we have definiteness of arithmetic objects across these worlds. But since these models do not all agree on their accounts of arithmetic truth, they violate the necessary requirement for the definiteness of truth. For this reason, the situation constitutes a counterexample to the Objects$\to$Truth definiteness principle.

We should like furthermore to emphasize that, for all we know, the situation of this example could be the actual state of affairs for our own mathematical ontology in regard to our arithmetic and set-theoretic concepts. For all we know, after all, we might be amongst the people living happily in such a world as $M_1$, taking the set-theoretic ontology of that world as our one true set-theoretic background---there seems no way in principle that we could ever know this was not the case. But in this event, we would look upon $\<\N,+,\cdot,0,1,<>^{M_1}$ as the standard model of arithmetic, blissfully unaware that there is another possible world $M_2$ with exactly this same structure of arithmetic and yet having different arithmetic truths. For this reason, we would be in error to deduce the definiteness of our theory of arithmetic truth solely on the basis of the seeming definiteness of the mathematical structure in which that truth resides, as there would be other set-theoretic realms with exactly the same arithmetic structure and yet different arithmetic truths.

\subsection{Definiteness of truth as a higher-order commitment}

In the situation we have exhibited, how exactly does the indefiniteness about arithmetic truth arise? How is it possible that the two models agree on the natural-number structure and yet disagree on arithmetic truth? Let us explain a little more about it.  In order for this indefiniteness to occur, we claim, the two models $M_1$ and $M_2$ must have different second-order arithmetic. In particular, they will have different incompatible arithmetic truth predicates. Neither model will be able to see the arithmetic truth predicates of the other because both of these predicates obey the Tarskian truth conditions and one may easily prove by induction that any two predicates satisfying the Tarskian conditions must agree universally. So neither model can have access to the arithmetic truth predicate of the other, and consequently, the models cannot agree on the power set of the natural numbers; so they cannot agree on the real numbers. In this way, indefiniteness of truth requires indefiniteness in the ontology of second-order objects. 

This is why we say that the commitment to a definite theory of truth for a structure is a higher-order ontological commitment, one going strictly beyond the definiteness of the underlying structure itself. The assertion that there is definite arithmetic truth is a claim about the definiteness of certain second-order objects, asserting that there is no indefiniteness for the second-order object that constitutes the truth predicate fulfilling the Tarski truth conditions. For the definiteness of arithmetic truth, it is not enough to express a definite nature for the natural numbers $0$, $1$, $2$, and so on, but rather one seems to require a degree of definiteness for the real numbers as well. 

%\subsection{Definiteness as object theory meta-theory alignment}
%
%Skip this point? Not really directly related to our thesis
%
%The claim of definiteness for arithmetic truth amounts in a certain sense to a claim that one's meta-theoretic concept of natural number aligns with the natural number concept in the object theory. One may prove by induction in the meta-theory, after all, that the truth of any standard-finite arithmetic sentence (that is, finite with the respect to the meta-theoretic natural number concept) is invariant with respect to any change in the set-theoretic background having the same natural number objects and structure. This is simply another way to look upon the fact that any two truth predicates on a given model of arithmetic must agree on their judgments for standard-finite formulas. Perhaps both Feferman and Martin, in their displayed quotations at the beginning of this section, would agree that the perceived clarity of the natural number concept leads them to have the same concept in both the meta-theory as in the object theory. And surely if one can be confident that one's meta-theoretic natural number concept coincides with the object-theoretic account, then one should expect definiteness of arithmetic truth. But what we would desire is an account of how this identity is supposed to work.

\subsection{The meaning of an assertion}

Roman Kossak has objected to our argument on the grounds that the arithmetic assertion $\sigma$, which has different truth values in the two models of set theory $M_1$ and $M_2$, nevertheless has different \emph{meanings} in these two structures. Although the two models agree on the natural number structure $\<\N,+,\cdot,0,1,<>^{M_1}=\<\N,+,\cdot,0,1,<>^{M_2}$, nevertheless, according to this objection, the meaning of the sentence $\sigma$ has changed. The proof of theorem \ref{Theorem.SameNDifferentTruths}, after all, shows that $M_1$ and $M_2$ are isomorphic by an isomorphism $\pi$ mapping a different sentence $\tau$ to $\sigma$, and so according to this objection, the sentence $\sigma$ carries in $M_2$ the same meaning that sentence $\tau$ carried in $M_1$, rather than the meaning carried by $\sigma$ itself in $M_1$. 
%Similarly, in corollary \ref{Corollary.SameStructuresDifferentTruths} statement (4) and theorem \ref{Theorem.FailureOfCollectionEvenOdd}, one might think that the meaning of the ordinal $\delta$ has changed in the move from $M_1$ to $M_2$.

Wherein does the meaning of a sentence lie? To our way of thinking, the meaning of a sentence $\sigma$ is best construed as precisely what that sentence expresses. And what the sentence expresses is determined completely by its compositional syntactic structure---the parse tree---for it is this compositional hereditary syntactic structure of the sentence that provides the instructions for how to calculate its truth, and this at bottom is what constitutes its meaning. The meaning of a sentence is precisely what the sentence expresses by its form.
%
%, particularly in the light of the fact that also $\N^{M_1}=\N^{M_2}$. In particular, all the fundamental syntactic features of $\sigma$ are the same in $M_1$ as in $M_2$, including the quantifier complexity, the scopes of quantifiers and the construction tree witnessing that $\sigma$ is a well-formed formula; these are all identical in $M_1$ as in $M_2$. Although we agree that the {\em truth} of an arithmetic sentence depends on the existence and nature of a second-order object---the arithmetic truth predicate---nevertheless the {\em meaning} of a sentence is more tightly connected with the syntactic nature of that sentence; the meaning is precisely what that sentence itself expresses. 

Tarski seems to agree that the meaning of a sentence is unambiguously determined by its syntactic form.
\begin{quote}
%§ 2. FORMALIZED LANGUAGES, ESPECIALLY THE LANGUAGE OF THE CALCULUS OF CLASSES
%
%For the reasons given in the preceding section I now abandon the attempt to solve our problem for the language of everyday life and
restrict[ing] myself henceforth entirely to formalized languages. These can be roughly characterized as artificially
% 1 The antinomy of heterological words (which I shall not describe here-cf. Grelling, K., and Nelson, L. (24), p. 307) is simpler than the antinomy of the liar in so far as no empirical premise analogous to (2) appears in its formula-tion; thus it leads to the correspondingly stronger consequence: there can be no consistent language which contains the ordinary laws of logic and satisfies two conditions which are analogous to (I) and (II), but differ from them in that they treat not of sentences but of names and not of the truth of sentences but of the relation of denoting. In this connection compare the discussion in § 5 of the present article -the beginning of the proof of Th. 1, and in particular p. 248, footnote 2.
%
% 2 The results obtained for formalized language also have a certain validity for colloquial language and that as owing to its universality: if we translate into colloquial language any definition of a true sentence which has been con-structed for some formalized language, we obtain a fragmentary definition of truth which embraces a wider or narrower category of sentences.
%
constructed languages in which the sense of every expression is unambiguously determined by its form. 
%Without attempting a completely exhaustive and precise description, which is a matter of considerable difficulty, I shall draw attention here to some essential properties which all the formalized languages possess: ($\alpha$) for each of these languages a list or description is given in structural terms of all the \emph{signs with which the expressions of the language are formed;} ($\beta$) among all possible expressions which can be formed with these signs those called sentences are distinguished by means of purely structural properties. Now formalized languages have hitherto been constructed exclusively for the purpose of studying deductive sciences formalized on the basis of such languages. The language and the science grow together to a single whole, so that we speak of the language of a particular formalized deductive science, instead of this or that formalized language. For this reason further characteristic properties of formalized languages appear in connexion with the way in which deductive sciences are built    
\qquad Alfred Tarski \cite[\S2]{Tarski1935:TheConceptOfTruthInFormalizedLanguages}
\end{quote}
According to this view, the meaning of an assertion is determined by what the assertion expresses, and thus is grounded ultimately in the exact nature of the parse tree of the assertion.

Since the models $M_1$ and $M_2$ agree on the arithmetic structure of the natural numbers, the sentence $\sigma$ has an identical parse tree structure in these models. And so the meaning of $\sigma$ in $M_1$ is precisely the same as the meaning of $\sigma$ in $M_2$. 

In particular, it is not the meaning of $\sigma$ that changed in the move from $M_1$ to $M_2$, we argue, but rather merely the truth judgment of $\sigma$ that has changed. There is of course an analogy between $\sigma$ in $M_2$ and $\tau$ in $M_1$, an analogy that is made precise and robust by the fact that they are automorphic images of each other. In the end, we therefore find Kossak's point to be not an objection to the argument, but rather an explanation of how it could come to be that arithmetic truth can vary without the objects and structure of arithmetic changing. 

Similarly, to argue in the context of theorem \ref{Theorem.SatisfyingZFCNotAbsolute} that $\delta$ refers in $M_2$ to something other than $\delta$ itself would seem to be a very hard task for Kossak. He replies to this by asking whether there is any meaning of $\delta$ at all other than that given by the semantics of the ambient universe? We believe that there is, since $\delta$ is a particular ordinal that both $M_1$ and $M_2$ have in common, and they agree on the cumulative hierarchy $V_\delta$ up to stage $\delta$, and one may inquire about the properties of this common object in $M_1$ and $M_2$. Our point is that because $M_1$ and $M_2$ agree on $V_\delta$, they will be forced to have a certain level of agreement of their theories of truth for this structure, but this agreement is not necessarily universal and ultimately $M_1$ and $M_2$ can disagree on whether certain statements are true in that common structure.

% The truth value of the assertion is thus not the same as its meaning, but rather is a higher-order judgement about the meaning---either that it is true or that it is not true---and this higher-order judgement is calculated in a second-order context according to the recursive definition of the satisfaction relation.

%In the models of our main theorem, the two worlds $M_1$ and $M_2$, we have an arithmetic assertion $\sigma$ that both worlds look upon as having exactly the same parse tree and thus this sentence expresses exactly the same meaning in both worlds, but the truth judgments about $\sigma$ in $M_1$ and $M_2$ are not the same. The two models agree on the first-order structure of $\<\N,+,\cdot,0,1,<>^{M_1}=\<\N,+,\cdot,0,1,<>^{M_2}$; they agree on the fact that $\sigma$ is an arithmetic sentence and they agree on the parse tree of this sentence; so they agree on the meaning of this sentence; but they disagree on the judgement of whether this sentence is true.

%How can this be? Well, the two models have different second-order ontology for arithmetic, and the satisfaction relation of $M_1$ is simply not present in $M_2$ and vice versa. Because of their differing and fundamentally incompatible solutions of the Tarski recursion, they come to different conclusions about whether $\sigma$ is true. This directly illustrates our contention that arithmetic truth is at bottom a higher-order commitment, which does not flow merely from agreement on the individuals and atomic judgments concerning a mathematical structure.

%\subsection{A better countermodel}

In connection with this objection, Kossak asked whether there is another method of producing examples where a given model of arithmetic has alternative truth predicates, and specifically, whether two models of set theory can have the same natural numbers, but non-isomorphic theories of arithmetic truth. That is, the question is whether one can have models for which $$\<\N,{+},{\cdot},0,1,{\lt}>^{M_1}=\<\N,{+},{\cdot},0,1,{\lt}>^{M_2}\text{, but}$$ $$\<\N,{+},{\cdot},0,1,<,\TA>^{M_1}\not\cong\<\N,{+},{\cdot},0,1,<,\TA>^{M_2}.$$
Jim Schmerl observed that indeed this situation occurs, arguing essentially as follows: the theory $T=\TA+\Th(\N,\TA)^\ZFC$, in the language with a satisfaction predicate $\Tr$, is computable from $0^{(\omega)}$, but the theory of $\<\N,{+},{\cdot},0,1,\lt,\TA>$ has degree $0^{(\omega+\omega)}$, and so $T$ is not complete in the expanded language. It follows that there are computably saturated models $\<\mathcal{N},\Tr_1>$ and $\<\mathcal{N},\Tr_2>$ of $T$, which have the same reduct to the language of arithmetic, but which are not elementarily equivalent in the full language and hence also not isomorphic. But since these models are computably saturated models of $\Th(\N,\TA)^\ZFC$, it follows by proposition \ref{Proposition.ZFCStandardTA} that they are both \ZFC-standard, arising as the standard model of arithmetic and arithmetic truth $\<\N,\TA>^{M_1}$ and $\<\N,\TA>^{M_2}$ inside models of \ZFC, as desired. This kind of example therefore seems to address Kossak's objection on the meaning of $\sigma$, while still establishing our main point, for in this case there is no automorphism providing one with an alternative meaning for the sentence $\sigma$ on which the models disagree.

\subsection{A spectrum of views on pluralism}

The question of definiteness of objects and truth is deeply connected with the issue of pluralism in mathematics and set theory, for pluralism itself is a kind of indefiniteness. Peter Koellner \cite{Koellner:HamkinsOnTheMultiverse} has emphasized that there is a hierarchy of positions to take on pluralism in mathematics, depending on where one expects indeterminism to first arise in mathematics, if it does so at all.
\begin{quote}
But, in fact, many people are non-pluralists with regard to certain branches of mathematics and pluralists with regard to others. For example, a very popular view in the foundations of mathematics embraces non-pluralism for first-order number theory while defending pluralism for the higher reaches of set theory. Most people would, for example, maintain that the Riemann Hypothesis (which is equivalent to a $\Pi^0_1$-statement of arithmetic) has a determinate truth-value and yet there are many who maintain that \CH\ (a statement of set theory, indeed, of third-order arithmetic) does not have a determinate truth-value. Feferman is an example of someone who holds this position. So instead of two positions---pluralism versus non-pluralism---we really have a hierarchy of positions. \cite{Koellner:HamkinsOnTheMultiverse}
\end{quote}
Summarizing the spectrum, on one end we have what might be described as the hard-core set-theoretic Platonists, such as Isaacson \cite{Isaacson2011:TheRealityOfMathematicsAndTheCaseOfSetTheory} (and to a more qualified extent, Woodin and Maddy \cite{Maddy1988:BelievingTheAxiomsI, Maddy1988:BelievingTheAxiomsII}), who affirm the existence of the universe of all sets, in which set-theoretic claims such as \CH\ have a definite truth value that we might come to know. Steel softens this position by allowing a restricted pluralism of mutually interpretable theories, such as occurs in the interaction of large cardinals, inner model theory and the theory of the axiom of determinacy. Martin \cite{Martin2001:MultipleUniversesOfSetsAndIndeterminateTruthValues} considers the case for determinism in set theory, but declining a full endorsement, argues instead only that there is at most one concept of set meeting his criteria, leaving explicitly open the possibility that in fact there is none. In the center of the spectrum we are considering, Feferman \cite{Feferman1999:DoesMathematicsNeedNewAxioms?} is committed to the Platonic existence of the natural numbers, while remaining circumspect about the definiteness of higher-order set-theoretic objects and truth. He asserts that ``the origin of Dedekind-Peano axioms is a clear intuitive concept,'' while the intuition for set theory ``is a far cry from what leads one to accept the Dedekind-Peano axioms,'' and based on the dichotomy between the ontological status of the structure of natural numbers and of the structure of higher types, he concludes that ``the Continuum Hypothesis is an inherently vague problem that no new axiom will settle in a convincingly definite way.'' Nik Weaver~\cite{Weaver2005:Predicativity-beyond-Gamma0} similarly calls for classical logic in the realm of arithmetic, where he takes mathematical assertions to have a clear and distinct definite nature; but in the set-theoretic realm of sets of numbers or sets of sets of numbers, which he takes to have a less definite nature, we should be using intuitionistic logic.

Meanwhile, Hamkins \cite{Hamkins2015:IsTheDreamSolutionToTheContinuumHypothesisAttainable,Hamkins2014:MultiverseOnVeqL, Hamkins2011:TheMultiverse:ANaturalContext, Hamkins2012:TheSet-TheoreticalMultiverse, GitmanHamkins2010:NaturalModelOfMultiverseAxioms} advances a more radical multiverse perspective that embraces pluralism even with respect to arithmetic truth. He criticizes the categoricity arguments for arithmetic definiteness---which are also used to justify higher-order definiteness, as in Isaacson \cite{Isaacson2011:TheRealityOfMathematicsAndTheCaseOfSetTheory}---as unsatisfactorily circular, in that they attempt to establish the definiteness of the natural number concept by appeal to second-order features involving a consideration of arbitrary subsets, which we would seem to have even less reason to think of as definite. More extreme positions, such as ultrafinitism, lie further along the spectrum.

\subsection{The nonstandardness objection}

It would be natural to object to our argument using $M_1$ and $M_2$ on the grounds that these models are $\omega$-nonstandard, and furthermore, every instance of this kind of non-determinism will be with models whose natural numbers are non-standard. Such an argument, however, seems to us to beg the question, since the issue is whether definiteness of truth follows from definiteness of objects and structure, not whether it follows from definiteness of objects and structure and the knowledge that those objects are ``standard'' with respect to some meta-mathematical conception of standardness. One would want to know, for example, about whether that conception of standardness is itself definite, for it is easy enough to establish that the concept of whether a given model of arithmetic is standard or not is not absolute between all models of set theory.

The natural number structure we consider, after all, is precisely the standard model of arithmetic from the point of view both of $M_1$ and $M_2$, since it is $\<\N,+,\cdot,0,1,<>^{M_1}$, which is identical to $\<\N,+,\cdot,0,1,<>^{M_2}$. So the structure we consider is already ``standard'' from the point of view of the set-theoretic models in which we consider it. It is only seen as nonstandard from a larger set-theoretic conception able to encompass both $M_1$ and $M_2$.

For our opponents to augment their argument for definite truth by making additional claims about the special nature of the natural numbers they consider, going beyond mere definiteness to a larger meta-theoretic notion of standardness, merely serves to verify our main point, which is that indeed something more needs to be said about it.

\subsection{Conclusion}

Ultimately, we argue that the definiteness of arithmetic truth does not flow for free from the definiteness of the numbers and the arithmetic structure $\<\N,+,\cdot,0,1,<>$, but rather amounts to a higher-order level of definiteness, for which one must argue separately. Even if two mathematicians might agree on the ontology of the numbers $0, 1, 2,\ldots$ and on their basic arithmetic structure, they may disagree on the judgments of arithmetic truth, if they do not agree on the second-order theory, as the models $M_1$ and $M_2$ in our main theorem show. A philosopher cannot legitimately claim to have a clear and definite conception of the underlying structure of arithmetic as a reason to take arithmetic truth also as definite.

Perhaps Feferman and the others will reply to the argument we have given in this article by saying that their conception of the natural numbers $0,1,2,\ldots$ and the natural number structure is so clear and distinct that they can use this overwhelming definiteness to see that arithmetic truth must also be definite. That is, perhaps they reply to our objection by saying that they did not base the definiteness of arithmetic truth solely on the fact that the natural number objects and structure are definite, but rather on something more, on the special nature of the definiteness of natural number objects and structure that they assert is manifest. For example, perhaps they would elaborate by explaining that arithmetic truth must be definite, because any  indefiniteness would reveal a proper cut in the natural numbers: consider the natural numbers $n$ for which $\Sigma_n$ truth is definite and argue that this contains $0$ and is closed under successors.  

That would be fine. Our response to such a reply, after first getting in the obvious objection about the indefinite nature of definiteness and whether it can be used in a mathematical induction in this way, would be to say that yes, indeed, this kind of further discussion and justification is exactly what we call for. Our main point, after all, is that definiteness-about-truth does not follow for free from definiteness-about-objects and that one must say something more about it.

\bibliographystyle{alpha}
\bibliography{MathBiblio,HamkinsBiblio,PhilBiblio,WebPosts}
\end{document}